\documentclass[11pt,reqno]{amsart}
\usepackage{amsbsy,amsfonts,amsmath,amsopn,amssymb,amsthm,amsxtra}
\usepackage{bm,bbm,color,enumerate,float,graphicx,ifthen}
\usepackage[inline,shortlabels]{enumitem}
\usepackage{mathrsfs,multirow,pbox,pgf,tikz,tikz-cd,tikz-3dplot,upref,url,xcolor,wrapfig}
\usepackage[utf8]{inputenc}
\usepackage[english]{babel}
\usepackage[all,color]{xy}
\usepackage[backref=page,linktocpage]{hyperref}
\hypersetup{
	colorlinks,
	final,
	linkcolor = blue1,
	citecolor = magenta,
}
\newcommand{\defn}[1]{{\color{blue2}\bf#1}}
\usepackage[letterpaper,margin = 1in]{geometry}
\usepackage[capitalize,noabbrev,nosort]{cleveref}
\crefname{equation}{}{}
\usepackage[square,sort,comma,numbers]{natbib}
\usepackage{doi}
\usetikzlibrary{shapes.multipart,positioning,decorations.pathreplacing,calc,intersections,folding,patterns}

\newtheorem{theorem}{Theorem}[section]
\newtheorem{proposition}[theorem]{Proposition}
\newtheorem{lemma}[theorem]{Lemma}
\newtheorem{corollary}[theorem]{Corollary}
\newtheorem{conjecture}[theorem]{Conjecture}
\newtheorem*{theorem*}{Theorem}
\theoremstyle{definition}
\newtheorem{definition}[theorem]{Definition}

\newtheorem{example}[theorem]{Example}

\theoremstyle{remark}
\newtheorem{remark}[theorem]{Remark}

\theoremstyle{plain}

\definecolor{blue1}{HTML}{007CFC}
\definecolor{blue2}{HTML}{4169E1}
\definecolor{red1}{HTML}{D20000}
\newcommand\blue[1]{{\color{blue1}{#1}}}
\newcommand\red[1]{{\color{red1}{#1}}}

\newcommand{\oeis}[1]{\href{https://oeis.org/#1}{#1} in the OEIS \citep{oeis}}

\newcommand{\cham}{\mathcal{R}}
\newcommand{\face}{\Sigma}
\renewcommand{\flat}{\mathcal{L}}
\DeclareMathOperator{\codim}{codim}
\DeclareMathOperator{\rank}{rank}
\newcommand{\arr}{\mathcal{A}}
\newcommand{\barr}{\mathcal{B}}
\newcommand{\darr}{\mathcal{D}}
\newcommand{\cB}{\mathcal{B}}
\newcommand{\cD}{\mathcal{D}}
\newcommand{\cG}{\mathcal{G}}
\newcommand{\cI}{\mathcal{I}}
\newcommand{\cP}{\mathcal{P}}
\newcommand{\cQ}{\mathcal{Q}}
\newcommand{\frakB}{\mathfrak{B}}
\newcommand{\frakD}{\mathfrak{D}}
\newcommand{\frakS}{\mathfrak{S}}
\newcommand{\frakz}{\mathfrak{z}}
\newcommand{\rmH}{\mathrm{H}}
\newcommand{\rmL}{\mathrm{L}}
\newcommand{\rmX}{\mathrm{X}}
\newcommand{\rmY}{\mathrm{Y}}
\newcommand{\rmZ}{\mathrm{Z}}
\newcommand{\sfh}{\mathsf{h}}
\newcommand{\RR}{\mathbb{R}}
\newcommand{\ZZ}{\mathbb{Z}}
\newcommand{\opp}[1]{\overline{#1}} 
\newcommand{\qqand}{\qquad\text{and}\qquad}
\newcommand{\qqiff}{\qquad\text{if and only if}\qquad}
\DeclareMathOperator{\asc}{asc} 
\DeclareMathOperator{\Des}{Des} 
\DeclareMathOperator{\des}{des} 
\DeclareMathOperator{\exc}{exc} 
\DeclareMathOperator{\fexc}{fexc} 
\DeclareMathOperator{\fdes}{fdes} 
\DeclareMathOperator{\fneg}{neg} 

\title[The Primitive Eulerian polynomial]{The Primitive Eulerian polynomial}

\address{
LACIM\\
Universit\'e du Qu\'ebec \`a Montr\'eal\\
CP 8888 Succ. Centre-Ville\\
Montr\'eal, Qu\'ebec, H3C 3P8\\ Canada}

\author[J.~Bastidas]{Jose Bastidas}
\address[Jose Bastidas]{}
\email{bastidas\_olaya.jose\_dario@uqam.ca}
\urladdr{\url{https://sites.google.com/view/bastidas}}

\author[C.~Hohlweg]{Christophe~Hohlweg}
\address[Christophe Hohlweg]{}
\email{hohlweg.christophe@uqam.ca}
\urladdr{\url{http://hohlweg.math.uqam.ca}}

\author[F.~Saliola]{Franco Saliola}
\address[Franco Saliola]{}
\email{saliola.franco@uqam.ca}
\urladdr{\url{https://saliola.github.io/}}



\date{\today}

\thanks{This research was partially supported by the NSERC grant ``Algebraic and geometric combinatorics of Coxeter groups'' held by CH}

\begin{document}

\begin{abstract}
We introduce the Primitive Eulerian polynomial $P_\arr(z)$ of a central hyperplane arrangement $\arr$.
It is a reparametrization of its cocharacteristic polynomial.
Previous work of the first author implicitly show that, for simplicial arrangements, $P_\arr(z)$ has nonnegative coefficients.
For reflection arrangements of type A and B, the same work interprets the coefficients of $P_\arr(z)$ using the (flag)excedance statistic on (signed) permutations.
The main result of this article is to provide an interpretation of the coefficients of $P_\arr(z)$ for all simplicial arrangements only using the geometry and combinatorics of $\arr$.

This new interpretation sheds more light to the case of reflection arrangements and, for the first time, gives combinatorial meaning to the coefficients of the Primitive Eulerian polynomial of the reflection arrangement of type D.
In type B, we find a connection between the Primitive Eulerian polynomial and the $1/2$-Eulerian polynomial of Savage and Viswanathan (2012).
We present some real-rootedness results and conjectures for $P_\arr(z)$.
\end{abstract}

\maketitle

{\setcounter{tocdepth}{1}\parskip=0pt\footnotesize\tableofcontents}

\section*{Introduction}

\renewcommand{\thetheorem}{\Alph{theorem}}

Let $\arr$ be a finite linear hyperplane arrangement in $\RR^n$.
We denote by $\flat$ its intersection lattice ordered by inclusion and $\bot \in \flat$ its minimum element (i.e. the intersection of all the hyperplanes in $\arr$).
The aim of this this article is to study the \defn{Primitive Eulerian polynomial} of $\arr$
\[
	P_\arr(z) := \sum_{\rmX \in \flat} |\mu(\bot,\rmX)| (z-1)^{\codim(\rmX)},
\]
where $\mu$ is the M\"obius function of $\flat$.

The Primitive Eulerian polynomial appeared implicitly in previous work of the first author~\citep{bastidas20polytope}, but only in the simplicial case.
Concretely, $P_\arr(z)$ is the Hilbert-Poincare series of a certain graded subspace of the \emph{polytope algebra of generalized zonotopes of $\arr$} \citep{mcmullen89,mcmullen93simple},
and therefore has non-negative coefficients.
It can also be obtained as a reparametrization of the cocharacteristic polynomial studied by Novik, Postnikov, and Sturmfels in \citep{nps02syzygies} which, conjecturally, cannot be obtained as a specialization of the well-known Tutte polynomial.

The main result of this article is to provide an explanation for the non-negativity of these coefficients from the geometry and combinatorics of $\arr$, more precisely, via \emph{generic halfspaces} and the \emph{weak order} on the regions of $\arr$.
Furthermore, for reflection arrangements of types A, B, and D, we give a combinatorial interpretation of these coefficients in terms of the usual \emph{Coxeter descent} statistic.

More precisely, let $\cham$ denote the collection of \defn{regions of $\arr$}: the (closures of) connected components of the complement of $\arr$ in $\RR^n$.
For any two regions $C$ and $C'$, ${\rm sep}(C,C')$ is the set of hyperplanes $\rmH \in \arr$ that separate $C$ and $C'$;
that is, such that $C$ and $C'$ are not contained in the same halfspace bounded by $\rmH$.
Fix a \defn{base region} $B \in \cham$.
The \defn{weak order with base region~$B$}~\citep{mandel82thesis,edelman84regions} is the partial order relation $\preceq_B$ on $\cham$ defined by
\[
	C \preceq_B C'
	\qqiff
	{\rm sep}(B,C) \subseteq {\rm sep}(B,C').
\]
Let $\des_{\preceq_B}(C)$ denote the number of regions covered by $C$ in the partial order $\preceq_B$.

A hyperplane $\rmH$ is \defn{generic with respect to $\arr$} if it contains $\bot$ and does not contain any other flat of $\arr$.
A halfspace $\sfh$ bounded by such a hyperplane is also said to be \defn{generic with respect to~$\arr$}.
A vector $v \in \RR^n$ is \defn{very generic with respect to $\arr$} if it is not contained in any hyperplane of $\arr$ and the halfspace $\sfh_v^- := \{ x \in \RR^n \,:\, \langle x , v \rangle \leq 0 \}$ is generic with respect to $\arr$.

A simplicial arrangement $\arr$ is \defn{sharp} if for all regions $C \in \cham$, the angle between any two facets of $C$ is at most $\tfrac{\pi}{2}$. Notably, all
Coxeter
arrangements are sharp.

\begin{theorem}\label{thm:A}
Let $\arr$ be a sharp arrangement.
Then, for any very generic vector $v \in \RR^n$,
\[
	P_\arr(z) = \sum_{C \subseteq \sfh_v^-} z^{\des_{\preceq_{B(v)}}(C)},
\]
where $B(v) \in \cham$ denotes the unique region of $\arr$ containing $v$.
The sum is over all regions $C \in \cham$ contained in $\sfh^-_v$.
\end{theorem}

\begin{wrapfigure}[12]{r}{0.35\textwidth} \centering
\includegraphics[width = 0.28\textwidth]{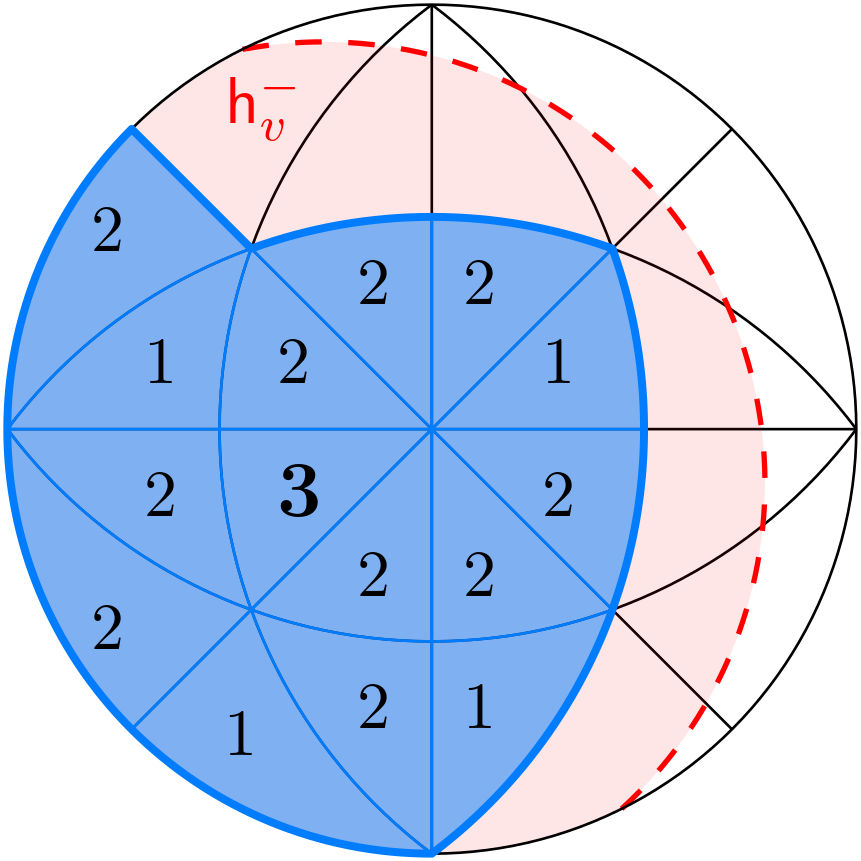}
\end{wrapfigure}

For a precise statement see \cref{t:sharp-distr}.
The figure on the right illustrates this theorem, as we now explain.
It shows
the intersection with the unit sphere of
a sharp arrangement~$\arr$ in~$\RR^3$.
The very generic vector $v \in \RR^3$ lies in the region antipodal to the region with label $3$, so it is not visible on this picture.
The blue regions $C$ are those contained in the halfspace $\sfh_v^-$, and their labels indicate the values $\des_{\preceq_{B(v)}}(C)$.
It follows that the Primitive Eulerian polynomial of this arrangement is 
$P_{\arr}(z) = z^3 + 10 z^2 + 4z$.

When $\arr$ is the reflection arrangement of type ${\rm A}_{n-1}$ (resp. ${\rm B}_n$) in $\RR^n$,
the first author provides in \citep{bastidas20polytope}
a combinatorial interpretation of the coefficients of $P_\arr(z)$ in terms of the excedance (resp. flag-excedance) statistic.
Concretely, let ${\rm cusp}(\frakS_n)$ (resp. ${\rm cusp}(\frakB_n)$) denote the collection of elements whose action on $\RR^n$ only fixes points in $\bot$, then
\[
	P_{\frakS_n}(z) = \sum_{w \in {\rm cusp}(\frakS_n)} z^{\exc(w)}
	\qquad\qqand\qquad
	P_{\frakB_n}(z) = \sum_{w \in {\rm cusp}(\frakB_n)} z^{\exc_B(w)}.
\]
The notation reflects the fact that these are precisely the \defn{cuspidal elements} of the corresponding Coxeter group \citep{gp00CharsCox}.
In $\frakS_n$, they are precisely the cycles of order $n$; and in~$\frakB_n$, they are the fixed-point-free products of balanced cycles.
See \cref{ss:PEul-A,ss:PEul-B} for details.
However, up to now, there was no combinatorial interpretation for the reflection arrangement of type D:
using \cref{t:sharp-distr} and building upon the work of Bj\"{o}rner and Wachs \citep{bw04geombases}, we provide such an interpretation.
More precisely, the Coxeter group of type ${\rm D}_n$ can be identified with the group of \defn{even signed permutations}: words $w = w_1 w_2 \dots w_n$ with $w_i \in [\pm n] = \{-n,\dots,-2,-1,1,2,\dots,n\}$ such that \hbox{$|w_1| |w_2| \dots |w_n| \in \frakS_n$}
and~$w$ has an even number of negative letters.
Let $BW^D_n$ be the collection of $w$ such that all the right-to-left maxima of $|w_1| |w_2| \dots |w_n| \in \frakS_n$ are negative in $w$ and $|w_1| \neq n$.
We obtain the following.

\begin{theorem}
For all $n \geq 2$,
\[P_{\cD_n}(z) = \sum_{w \in BW^D_n} z^{\des(w)}.\]
\end{theorem}
For a precise statement see \cref{t:PE-D-des}.
Analogous results for type A and B appear in \cref{t:PE-A-des,t:PE-B-des}.

The article is organized as follows.
In \cref{s:prelims} we review some preliminaries on hyperplane arrangements.
We introduce the Primitive Eulerian polynomial in \cref{s:PEul} and give a combinatorial interpretation for its coefficients, in the case of simplicial arrangements, in \cref{s:upsets}.
We specialize our result to the case of Coxeter arrangements in \cref{s:PEulCox};
closely examine the types A, B, and D cases in \cref{ss:PEul-A,ss:PEul-B,ss:PEul-D};
and compute the Primitive Eulerian polynomial for an infinite family of simplicial arrangements sitting between the type D and type B arrangements in \cref{ss:betweenDB}.
Finally, in \cref{s:roots} we explore some real-rootedness results and conjectures.

\subsection*{Acknowledgments}

We would like to thank Federico Ardila and Eliana Tolosa Villareal for helpful discussions and for kindly sharing their Master's Thesis with us.
We also thank Ron Adin and Yuval Roichman for pointing us to useful references on flag statistics on the hyperoctahedral group,
and Hugh Thomas and Nathan Reading for helpful comments which led to \cref{r:shell}.

\section{Preliminaries and notation}\label{s:prelims}

\renewcommand{\thetheorem}{\thesection.\arabic{theorem}}

We start by recalling some classical definitions, including restriction, localization and the Tits semigroup (or face semigroup) of a hyperplane arrangement; see for instance \citep{am17,stanley2007} for more details.
The reader familiar with these concepts can safely proceed to \cref{s:PEul}.

In this article we only consider \defn{finite real linear hyperplane arrangements} in $\RR^n$, i.e., a finite collection of linear hyperplanes $\rmH \subseteq \RR^n$.
The subspaces obtained by intersecting some hyperplanes in $\arr$ are called \defn{flats} of $\arr$.
The collection $\flat[\arr]$ of flats of $\arr$ is naturally ordered by inclusion.
It turns out that the poset $(\flat[\arr],\leq)$ is a graded lattice with maximum element $\top := \RR^n$
and minimum element $\bot := \bigcap \arr$. The arrangement $\arr$ is \defn{essential} if $\bot = \{0\}$.
Otherwise, we denote by ${\rm ess}(\arr)$ the \emph{essentialization} of $\arr$, i.e., the intersection of $\arr$ with the subspace orthogonal to $\bot$.

\subsection{Restriction and localization}

Let $\rmX \in \flat[\arr]$ be a flat of $\arr$.
The \defn{restriction} of $\arr$ to $\rmX$ is the following hyperplane arrangement inside $\rmX$:
\[
	\arr^\rmX
		= \{ \rmX \cap \rmH \,:\, \rmH \in \arr,\, \rmX \not\subseteq \rmH \}.
\]
Note that the hyperplanes of $\arr^\rmX$ are precisely the flats $\rmY < \rmX$ of dimension $\dim(\rmX) - 1$.
On the other hand, the \defn{localization} of $\arr$ at $\rmX$ is the arrangement in $\RR^n$ consisting of those hyperplanes of $\arr$ containing $\rmX$:
\[
\arr_\rmX = \{ \rmH \in \arr \,:\, \rmX \subseteq \rmH \}.
\]

\subsection{Regions and faces}

The complement of $\arr$ in $\RR^n$ is the disjoint union of open sets whose closures are convex polyhedral cones.
These full-dimensional cones are the \defn{regions} of $\arr$.
Denote by~$\cham[\arr]$ the collection of regions of $\arr$.
A \defn{face} of $\arr$ is any face of a cone in $\cham[\arr]$.
Denote by~$\face[\arr]$ the collection of faces of $\arr$, it forms a complete polyhedral fan.
Its maximal elements are the regions of $\arr$, and its unique minimal face, which coincides with the minimal flat $\bot$, is denoted $O \in \face[\arr]$.
Finally, the \defn{rank of $\arr$}, denoted $\rank(\arr)$, is the rank of the poset $\face[\arr]$ or, equivalently, the rank of the lattice $\flat[\arr]$.

The arrangement $\arr$ is \defn{simplicial} if every region contains exactly $\rank(\arr)$ faces of rank $1$.
If $\arr$ is essential, this is equivalent to each cone in $\face[\arr]$ being simplicial.

\subsection{The Tits product}

The set $\face[\arr]$ has the structure of a monoid under the \defn{Tits product}.
Informally, the product $FG$ of two faces $F,G \in \face[\arr]$ is the first face you enter when moving a small positive distance from a point in the relative interior of $F$ to a point in the relative interior of $G$, this is illustrated in \cref{f:Tits-prod}.
In order to formalize this product, it will be useful to review the \defn{sign sequence} of a face.

\begin{figure}[ht]
\[
	\begin{gathered}
	\begin{tikzpicture}[scale=.8]
		\draw [<->] (0+180:2) -- (0,0) -- (0:2);
		\draw [<->] (45+180:2) -- (0,0) -- (45:2);
		\draw [<->] (-45+180:2) -- (0,0) -- (-45:2);
		\draw [<->] (90+180:2) -- (0,0) -- (90:2);
		\draw [fill=black] (0,0) circle (.1);
		\node [left] at (0+180:2) {\small $F$};
		\node [below] at (-90:2) {\small $G$};
		\node at (22.5:1.25) {\small $C$};
		\node at (180-22.5:1.25) {\small $FC$};
		\node at (180+22.5:1.25) {\small $FG$};
		\node [below] at (-90+22.5:1.25) {\small $GC$};
		\draw [blue1,dotted] (180:1) -- (270:1);
		\draw [red,->] (180:1) --++ (-45:.4);
		\draw [blue1,fill=blue1] (0,-1) circle (.04);
		\draw [blue1,fill=blue1] (-1,0) circle (.04);
	\end{tikzpicture} 
	\end{gathered}
\]
\caption{The Tits product for some faces of a rank 2 arrangement. After moving a small positive distance from the marked point in $F$ to the marked point in $G$, we land in the region labeled $FG$.}
\label{f:Tits-prod}
\end{figure}
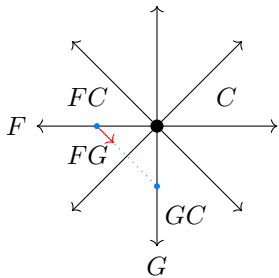

Each hyperplane $\rmH \in \arr$ defines two halfspaces $\rmH^+$ and $\rmH^-$. 
The choice of signs $+$ and $-$ for each hyperplane is arbitrary, but fixed.
The sign sequence of a face $F \in \face[\arr]$ is the vector $\sigma(F) \in \{0,+,-\}^\arr$ determined by
\[
	\sigma_\rmH(F) =
	\begin{cases}
		0 & \text{if } F \subseteq \rmH, \\
		+ & \text{if } F \subseteq \rmH^+ \text{ and } F \not\subseteq \rmH, \\
		- & \text{if } F \subseteq \rmH^- \text{ and } F \not\subseteq \rmH.
	\end{cases}
\]
Given two faces $F,G \in \face[\arr]$, the product $FG$ is the face with sign sequence
\begin{equation}\label{eq:Tits-product}
	\sigma_\rmH(FG) =
	\begin{cases}
		\sigma_\rmH(F) 		& \text{if } \sigma_\rmH(F) \neq 0, \\
		\sigma_\rmH(G) 		& \text{otherwise.}
	\end{cases}
\end{equation}
This operation first appeared in work of Tits \citep{tits74} on Coxeter complexes and of Bland \citep{bland74thesis} on oriented matroids.
Brown observed in \citep{brown00semigroups} that the
monoid $\face[\arr]$ is a \emph{left-regular band}, that is, it satisfies $F^2 = F$ and $FGF = FG$ for all faces $F,G \in \face[\arr]$.
We highlight two other important properties of the Tits product.
Let $O \in \face[\arr]$ be the central face of $\arr$ and $F,G \in \face[\arr]$ be any two faces. Then,
\begin{equation}\label{eq:Tits-properties}
	OF = F = FO
	\qquad\qquad
	F \leq FG
\end{equation}
That is, $\face[\arr]$ is a monoid with unit $O$, and $F$ is always a face of the product $FG$.

A hyperplane $\rmH$ \defn{separates} two faces $F$ and $G$ if $\{ \sigma_\rmH(F) , \sigma_\rmH(G) \} = \{+,-\}$.
That is, if $\rmH$ does not contain $F$ nor $G$, and $F$ and $G$ are contained in opposite halfspaces determined by $\rmH$.
Let $F \in \face[\arr]$, its \defn{opposite face} $\opp{F}$ is the face with sign sequence $\sigma_\rmH(\opp{F}) = -\sigma_\rmH(F)$. Geometrically, $\opp{F} = \{ -x \,:\, x \in F\}$.

\section{The Primitive Eulerian polynomial}\label{s:PEul}

We introduce the main definition of this article.
Let $\arr$ be a linear arrangement.

\begin{definition}\label{def:PEul}
The \defn{Primitive Eulerian polynomial} of $\arr$ is
\[
	P_\arr(z) = \sum_{\rmX \in \flat} |\mu(\bot,\rmX)| (z-1)^{\codim(\rmX)},
\]
where $\mu$ denotes the M\"obius function of $\flat$. The sum is over all the flats of the arrangement.
\end{definition}

The Primitive Eulerian polynomial can also be obtained as a reparametrization of
the \defn{cocharacteristic polynomial} $\Psi_\arr(z)$:
\begin{equation}\label{eq:cochar}
	\Psi_\arr(z) = \sum_{\rmX \in \flat} |\mu(\bot,\rmX)| z^{\dim(\rmX)}.
\end{equation}
Explicitly,
\begin{equation}\label{eq:cochar-to-PEul}
	P_\arr(z) = (z-1)^n \Psi_\arr\big( \dfrac{1}{z-1} \big).
\end{equation}
Novik, Postnikov, and Sturmfels encountered this polynomial in their study of unimodular toric arrangements and remarked that this polynomial does not seem to be a specialization of the Tutte polynomial of $\arr$.
If that is the case, then the Primitive Eulerian polynomial cannot be a specialization of the well-known Tutte polynomial either.

The Primitive Eulerian polynomial is monic: it's leading coefficient is $|\mu(\bot,\bot)| = 1$.
Moreover, if $\arr$ is not trivial (contains at least one hyperplane), then
\[
	P_\arr(0) = \sum_{\rmX \in \flat} |\mu(\bot,\rmX)| (-1)^{\codim(\rmX)}
	= (-1)^{\rank(\arr)} \sum_{\rmX \in \flat} \mu(\bot,\rmX) = 0.
\]
That is, the constant term of $P_\arr(z)$ is zero.

\begin{example}[Rank 1 arrangement]\label{ex:PEul-rk1}
Let $\arr$ be a hyperplane arrangement of rank $1$: it consists of a single hyperplane $\rmH$.
Then, $\flat[\arr] = \{ \bot < \top \}$, where $\bot = \rmH$, and
\[
	P_{\arr}(z) = (z-1) + 1 = z.
\]
\end{example}

\begin{remark}
Observe that $P_{\arr}(z)$ is independent of the dimension of the ambient space.
Indeed, for any arrangement $\arr$, the codimension of a flat $\rmX \in \flat[\arr]$ is precisely the corank of $\rmX$ in the lattice $\flat[\arr]$. Thus, the polynomial $P_{\arr}(z)$ is completely determined by $\flat[\arr]$.
In particular,
\[
	P_{\arr}(z) = P_{{\rm ess}(\arr)}(z).
\]
\end{remark}

\begin{example}[Rank 2 arrangements]\label{ex:PEul-rk2}
For $k \geq 2$, let $\cI_2(k)$ be the linear arrangement of $k$ lines in~$\RR^2$.
The M\"obius function of $\flat[\arr]$ is determined by
\[
	\mu(\bot,\rmL) = -1 \text{ for all lines } \rmL \in \arr,
	\qqand
	\mu(\bot,\top) = k-1.
\]
We obtain
\[
	P_{\cI_2(k)}(z) = (z-1)^2 + k (z-1) + (k-1) = z^2 + (k-2) z.
\]
\end{example}

It directly follows from the definition that the Primitive Eulerian polynomial of a Cartesian product of arrangements is the product of the corresponding Primitive Eulerian polynomials.
That is,
\begin{equation}\label{eq:prod-arr}
	P_{\arr \times \arr'}(z) = P_{\arr}(z) P_{\arr'}(z).
\end{equation}

Previous work of the first author implicitly shows that the polynomial $P_\arr(z)$ has nonnegative coefficients whenever the arrangement $\arr$ is simplicial.
We provide the details below.

\begin{proposition}For any simplicial arrangement $\arr$, the coefficients of $P_\arr(z)$ are nonnegative.
\end{proposition}

\begin{proof}
For every flat $\rmY \in \flat[\arr]$, let $\frakz_\rmY$ be a zonotope dual to $\arr_\rmY$.
That is, such that the normal fan of $\frakz_\rmY$ is $\face[\arr_\rmY]$.
In particular, the zonotope $\frakz_\rmY$ is a simple polytope.
Then, \citep[Equation (17)]{bastidas20polytope} shows that the polynomial
$\sum_\rmY \mu(\bot,\rmY) h(\frakz_\rmY,z)$,
where $h(\frakz_\rmY,z)$ denotes the $h$ polynomial of $\frakz_\rmY$,
has nonnegative coefficients. Manipulating \cref{def:PEul}, we obtain
\[
	P_\arr(z+1) = 
	\sum_{ \substack{\rmX,\rmY,\rmZ \\ \rmX \geq \rmZ \geq \rmY} } \mu(\bot,\rmY) |\mu(\rmZ,\rmX)| z^{\codim(\rmX)} =
	\sum_\rmY \mu(\bot,\rmY) f(\frakz_\rmY,z),
\]
where the second equality follows from Zaslavsky's \citep[Theorem A]{zaslavsky75facing} and Las Vergnas' \citep[Proposition 8.1]{lasvergnas75} formulas to count the number of regions of an arrangement.
Therefore,
\[
	P_\arr(z) = \sum_\rmY \mu(\bot,\rmY) f(\frakz_\rmY,z-1) = \sum_\rmY \mu(\bot,\rmY) h(\frakz_\rmY,z)
\]
has nonnegative coefficients.
\end{proof}

The nonnegativity of the coefficients of $P_\arr(z)$ fails for arbitrary arrangements, as the following example shows.

\begin{example}\label{ex:PEul-4cycle}
Let $\arr$ be the graphic arrangement of a $4$-cycle $g$, see \cref{f:flats4cycle} for an illustration.
The flats of $\arr$ are in correspondence with the \emph{bonds} of $g$, see for instance \citep[Section 6.9]{am17}.
The lattice $\flat[\arr]$ and the values $\mu(\bot,\rmX)$ of its Möbius function are also shown in \cref{f:flats4cycle}.
Using those values, we obtain
\[
P_\arr(z) = (z-1)^3+\red{6}(z-1)^2+\red{8}(z-1)+\red{3} = z^3 + 3 z^2 - z.
\]
Observe that the coefficient of the linear term is negative.
\end{example}

\begin{figure}[ht]
\[
\begin{tikzpicture}[scale = .7]
\newdimen\A
\A=1cm \newdimen\B
\B=1cm \newdimen\C
\C=2cm \node[inner sep=2pt] (1234) at    (0,0\C) {\small\phantom{$1$}\includegraphics[width=.5cm]{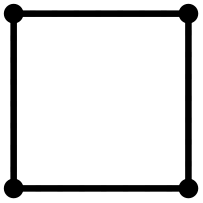}$\red{1}$};
\node[inner sep=2pt] (12)   at (-5\A,1\C) {\small\phantom{$-1$}\includegraphics[width=.5cm]{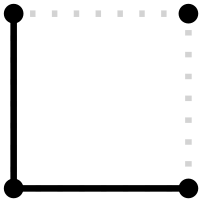}$\red{-1}$};
\node[inner sep=2pt] (13)   at (-3\A,1\C) {\small\phantom{$-1$}\includegraphics[width=.5cm]{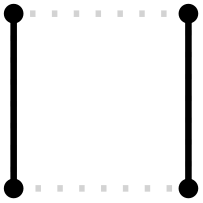}$\red{-1}$};
\node[inner sep=2pt] (14)   at  (-\A,1\C) {\small\phantom{$-1$}\includegraphics[width=.5cm]{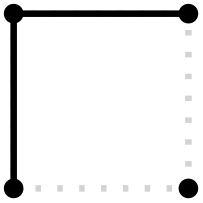}$\red{-1}$};
\node[inner sep=2pt] (23)   at   (\A,1\C) {\small\phantom{$-1$}\includegraphics[width=.5cm]{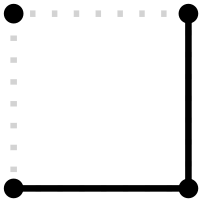}$\red{-1}$};
\node[inner sep=2pt] (24)   at  (3\A,1\C) {\small\phantom{$-1$}\includegraphics[width=.5cm]{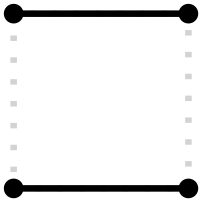}$\red{-1}$};
\node[inner sep=2pt] (34)   at  (5\A,1\C) {\small\phantom{$-1$}\includegraphics[width=.5cm]{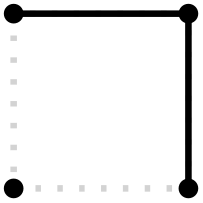}$\red{-1}$};
\node[inner sep=2pt] (1)    at (-3\B,2\C) {\small\phantom{$2$}\includegraphics[width=.5cm]{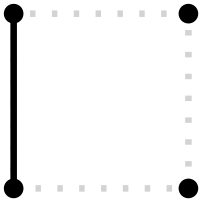}$\red{2}$};
\node[inner sep=2pt] (2)    at  (-\B,2\C) {\small\phantom{$2$}\includegraphics[width=.5cm]{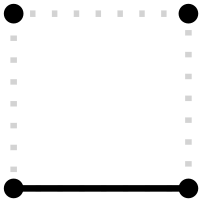}$\red{2}$};
\node[inner sep=2pt] (3)    at   (\B,2\C) {\small\phantom{$2$}\includegraphics[width=.5cm]{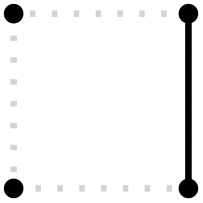}$\red{2}$};
\node[inner sep=2pt] (4)    at  (3\B,2\C) {\small\phantom{$2$}\includegraphics[width=.5cm]{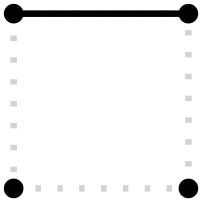}$\red{2}$};
\node[inner sep=2pt] (0)    at    (0,3\C) {\small\phantom{$-3$}\includegraphics[width=.5cm]{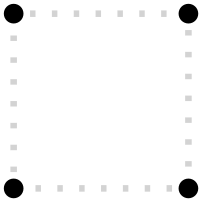}$\red{-3}$};
\draw[] (1234) -- (12) (1234) -- (13) (1234) -- (14) (1234) -- (23) (1234) -- (24) (1234) -- (34);
\draw[] (1) -- (12) -- (2);
\draw[] (1) -- (13) -- (3);
\draw[] (1) -- (14) -- (4);
\draw[] (2) -- (23) -- (3);
\draw[] (2) -- (24) -- (4);
\draw[] (3) -- (34) -- (4);
\draw[] (0) -- (1) (0) -- (2) (0) -- (3) (0) -- (4);
\end{tikzpicture}
\]
\caption{The lattice of flats of the graphic arrangement of a 4-cycle. In red, the values of the Möbius function $\mu(\bot,\rmX)$.}
\label{f:flats4cycle}
\end{figure}
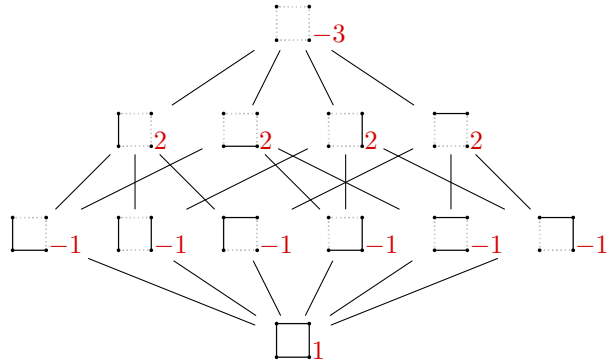

In view of the relation between the Primitive Eulerian polynomial and the cocharacteristic polynomial in \cref{eq:cochar-to-PEul}, the following recursive formula
is equivalent to \citep[Proposition 4.2]{nps02syzygies}.

\begin{proposition}\label{p:PEul-recursion}
Let $\rmH$ be a hyperplane of $\arr$. Then,
\[
P_{\arr}(z) = (z-1)P_{\arr^\rmH}(z) + \sum_\rmL P_{\arr_{\rmL}}(z),
\]
where the sum is over all rank 1 flats $\rmL \in \flat[\arr]$ that are not contained in $\rmH$.
\end{proposition}

In \cref{ss:PEul-A,ss:PEul-B,ss:PEul-D}, we the above result to obtain quadratic recursive formulas for the Primitive Eulerian polynomials of the classical reflection arrangements (types A, B, and D).

\section{Generic halfspaces, descents, and a combinatorial interpretation}\label{s:upsets}

In this section, we interpret the coefficients of $P_\arr(z)$ in combinatorial terms for any simplicial arrangement and therefore provide a proof of \cref{thm:A}.

\subsection{Generic halfspaces}

Let $\arr$ be a nontrivial hyperplane arrangement in $\RR^n$.
A hyperplane $\rmH \subseteq \RR^n$ (not in the arrangement) is \defn{generic with respect to $\arr$} if it contains the minimum flat $\bot$ and it does not contain any other flat of $\arr$.
The first condition guarantees that there is a canonical correspondence between generic hyperplanes with respect to $\arr$ and generic hyperplanes with respect to its essentialization ${\rm ess}(\arr)$.

\begin{figure}[ht]
\includegraphics[height=.2\textheight]{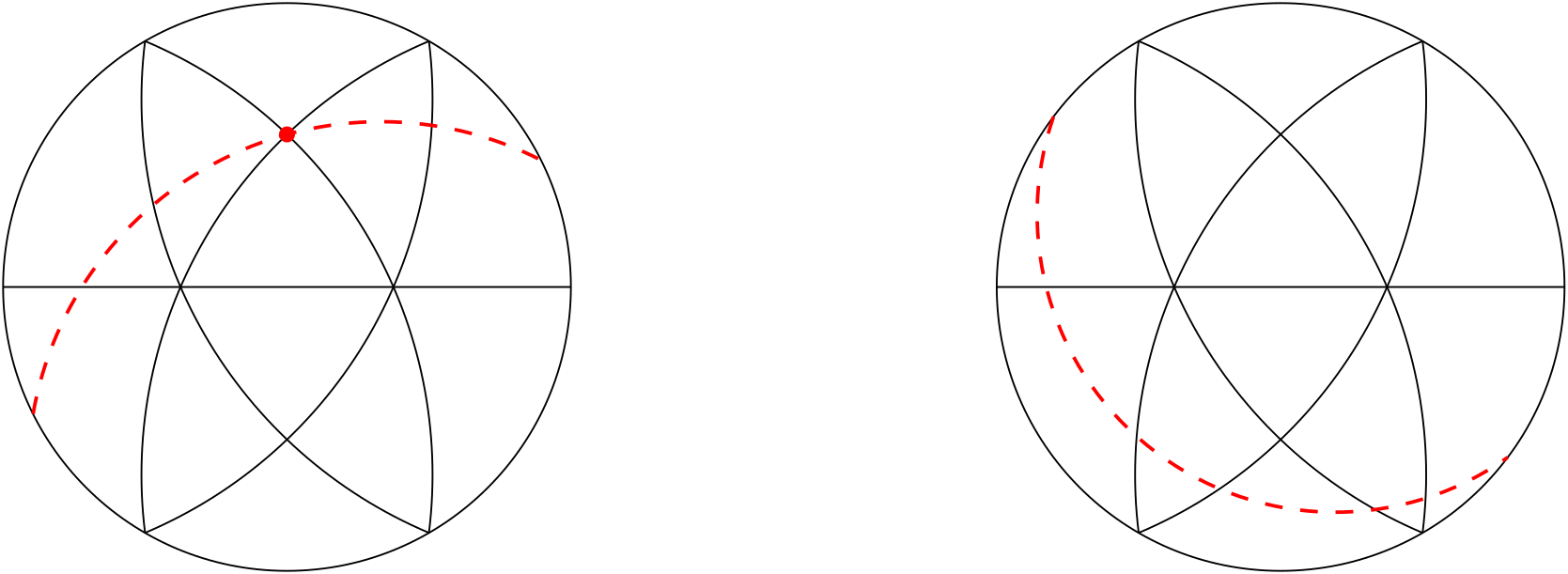}
\caption{Two copies of the spherical representation of an arrangement $\arr$ in $\RR^3$,
         each with a different hyperplane $\rmH$ not in $\arr$ (dashed, in red).
         The hyperplane on the left is not generic since it contains the marked flat of rank $1$ of $\arr$.
         The hyperplane on the right is generic.}
\label{f:gen-hyp}
\end{figure}

A halfspace $\sfh$ is \defn{generic with respect to $\arr$} if its bounding hyperplane is generic with respect to $\arr$.
Let $\sfh$ be such a halfspace.
A
result of Greene and Zaslavsky \citep[Theorem 3.2]{gz83whitney} shows that the number of regions $C \in \cham[\arr]$ completely contained in $\sfh$ is $|\mu(\bot,\top)|$.
Note that for all flats $\rmX$, the halfspace $\sfh \cap \rmX$ is generic with respect to the arrangement $\arr^\rmX$.
As a straightforward generalization of Greene and Zaslavsky's result, and in view of the definition of the cocharacteristic polynomial in \cref{eq:cochar}, we obtain:
\begin{equation}\label{eq:cochar-faces-arr}
	\Psi_\arr(z) = \sum_{F \subseteq \sfh} z^{\dim(F)},
\end{equation}
where the sum is over all the faces $F \in \face[\arr]$ that are contained in $\sfh$.

Given a subset $K$ of the ambient space, let $\face_\arr(K) \subseteq \face[\arr]$ and $\cham_\arr(K) \subseteq \cham[\arr]$ denote the collection of faces and regions, respectively, of $\arr$ contained in $K$.
When the arrangement $\arr$ is clear from the context, we drop $\arr$ from the notation and simply write $\face(K)$ and $\cham(K)$.
If the arrangement $\arr$ is simplicial, then $\face(K)$ can be viewed as a simplicial complex, where the vertices are the rays of $\arr$ and the central face of $\arr$ corresponds to the empty face of the simplicial complex.
We obtain the following geometric interpretation for the coefficients of $P_\arr(z)$ in the simplicial case.

\begin{figure}[ht]
\includegraphics[height=.2\textheight]{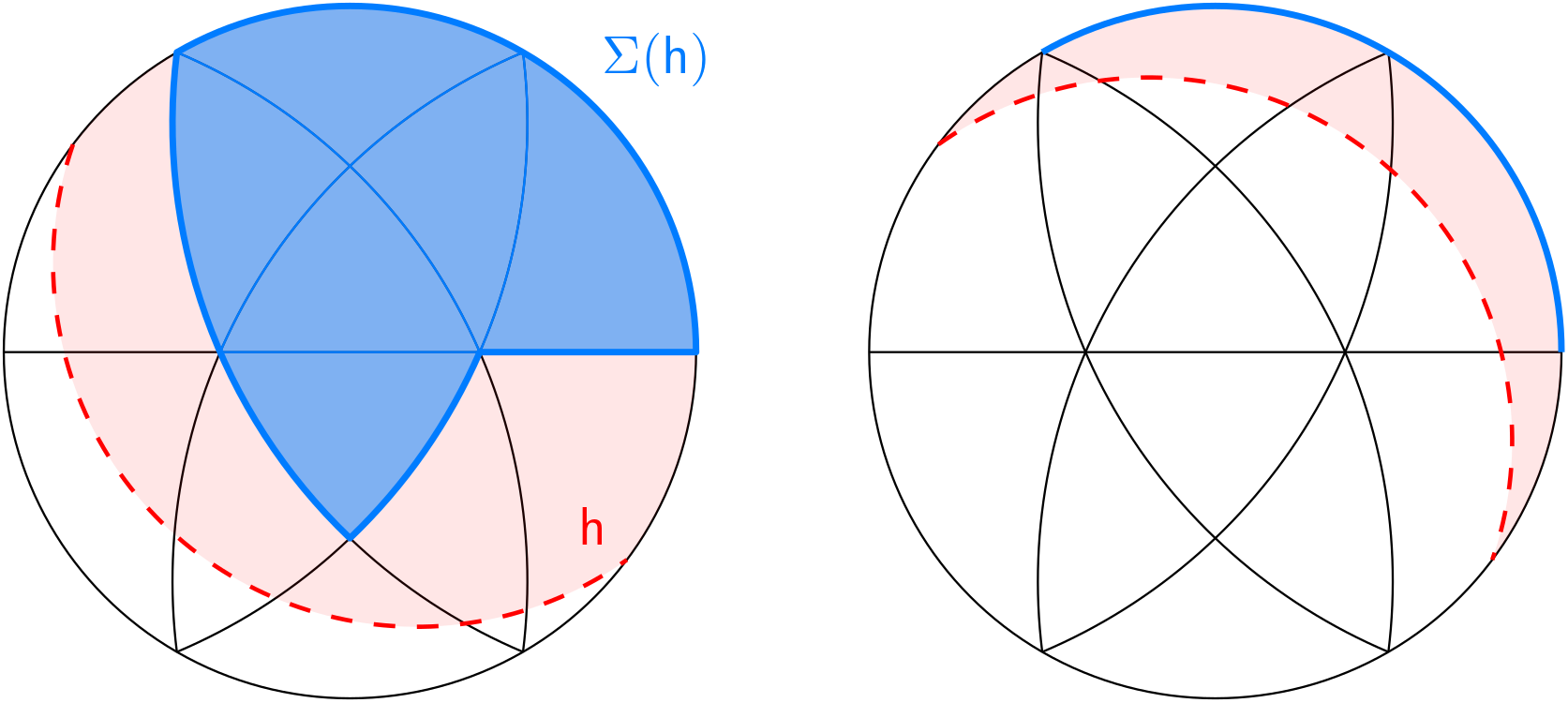}
\caption{The arrangement and the bounding hyperplane of $\sfh$ are those in second example of \cref{f:gen-hyp}.
         In this picture we see the \emph{front} and \emph{back} view of $\arr$.
         Observe that even though $\sfh$ is a convex set, the underlying set of $\face(\sfh)$ is not.
         We have that $\Psi_\arr(z) = 1 + 7z + 12 z^2 + 6 z^3$ and $P_\arr(z) = z^3 + 4 z^2 + z$.}
\label{f:not-convex}
\end{figure}

\begin{proposition}
Let $\arr$ be a simplicial arrangement and $\sfh$ a generic halfspace with respect to $\arr$.
Then,
\[
	P_\arr(z) = z^n h(\face(\sfh), \tfrac{1}{z}).
\]
\end{proposition}

\begin{proof}
Formula \cref{eq:cochar-faces-arr} is equivalent to $\Psi_\arr(z) = f(\face(\sfh), z)$.
Thus,
\begin{align*}
	z^n h(\face(\sfh), \tfrac{1}{z}) = 
	z^n (1 - \tfrac{1}{z})^n f \Big( \face(\sfh), \dfrac{1/z}{1 - 1/z} \Big) = 
	(z-1)^n \Psi_\arr\big( \dfrac{1}{z-1} \big).
\end{align*}
The result follows by \cref{eq:cochar-to-PEul}.
\end{proof}

\begin{remark}
Reiner \citep{reiner90thesis} and Stembridge \citep{stembridge08coxcones} studied $h$-vectors of complexes
of the form $\face(K)$, where $\arr$ is a Coxeter arrangement and $K$ is a \emph{cone} (parset) of the arrangement--a convex set obtained as the union of some regions of $\arr$.
The complexes obtained form a generic halfspace are in general not convex; see for instance \cref{f:not-convex}.
In that example, $\face(\sfh)$ is not convex, and therefore not a cone of the corresponding arrangement (the essentialization of the braid arrangement in $\RR^4$).
\end{remark}

\subsection{The weak order}

The definitions and results of this section hold for arbitrary linear arrangements, including non-simplicial arrangements.

Recall that a hyperplane $\rmH \in \arr$ \defn{separates} faces $F,G \in \face[\arr]$ if $\{ \sigma_\rmH(F) , \sigma_\rmH(G) \} = \{-,+\}$.
Given any two regions $C,C' \in \cham[\arr]$, let ${\rm sep}(C,C')$ denote the collection of hyperplanes $\rmH \in \arr$ that separate $C$ and $C'$.
Fix a region $B \in \cham[\arr]$ which we call a \defn{base region}, and
consider the following relation on $\cham[\arr]$:
\[
	C \preceq_B D
	\qqiff
	{\rm sep}(B,C) \subseteq {\rm sep}(B,D).
\]
Then, $\preceq_B$ is a partial order with minimum element $B$ and maximum element $\opp{B}$, the region {opposite} to the base region $B$.
This order is called the \defn{weak order of $\arr$ with base region $B$},
and was independently introduced by Mandel \citep{mandel82thesis} and Edelman \citep{edelman84regions}.
Mandel's definition was given in the context of oriented matroids, where it was called the \emph{Tope graph} of $\arr$.
This order was further studied by Bj\"{o}rner, Edelman and Ziegler \citep{bez90hyperplane}, who showed that whenever $\arr$ is simplicial, the weak order with respect to any base region is a lattice.
When the base region $B$ is clear from the context, we write $\preceq$ instead of $\preceq_B$.

For any face $F \in \face[\arr]$, the collection of regions that contain it,
\[
	\cham_F := \{C \in \cham[\arr] \,:\, F \leq C \},
\]
is called the \defn{top-star} of $F$.
See \cref{f:top-stars} for an example.

\begin{lemma}
For any base region $B$ and face $F$, the top-star $\cham_F$ is an interval in the partial order~$\preceq_B$:
its minimum element is $FB$ and its maximum element is $F\opp{B}$.
\end{lemma}

\begin{proof}
The result is a direct consequence from the next two observations, which follow from the Tits product description in \cref{eq:Tits-product}.
First, if $\rmH \in \arr$ separates $B$ and $F$, then it also separates $B$ and any region in $\cham_F$.
Moreover, any hyperplane $\rmH \in \arr$ containing $F$ does not separate $B$ and $FB$, and does separate $B$ and $F\opp{B}$.
\end{proof}

\begin{figure}[ht]
\includegraphics[height=.2\textheight]{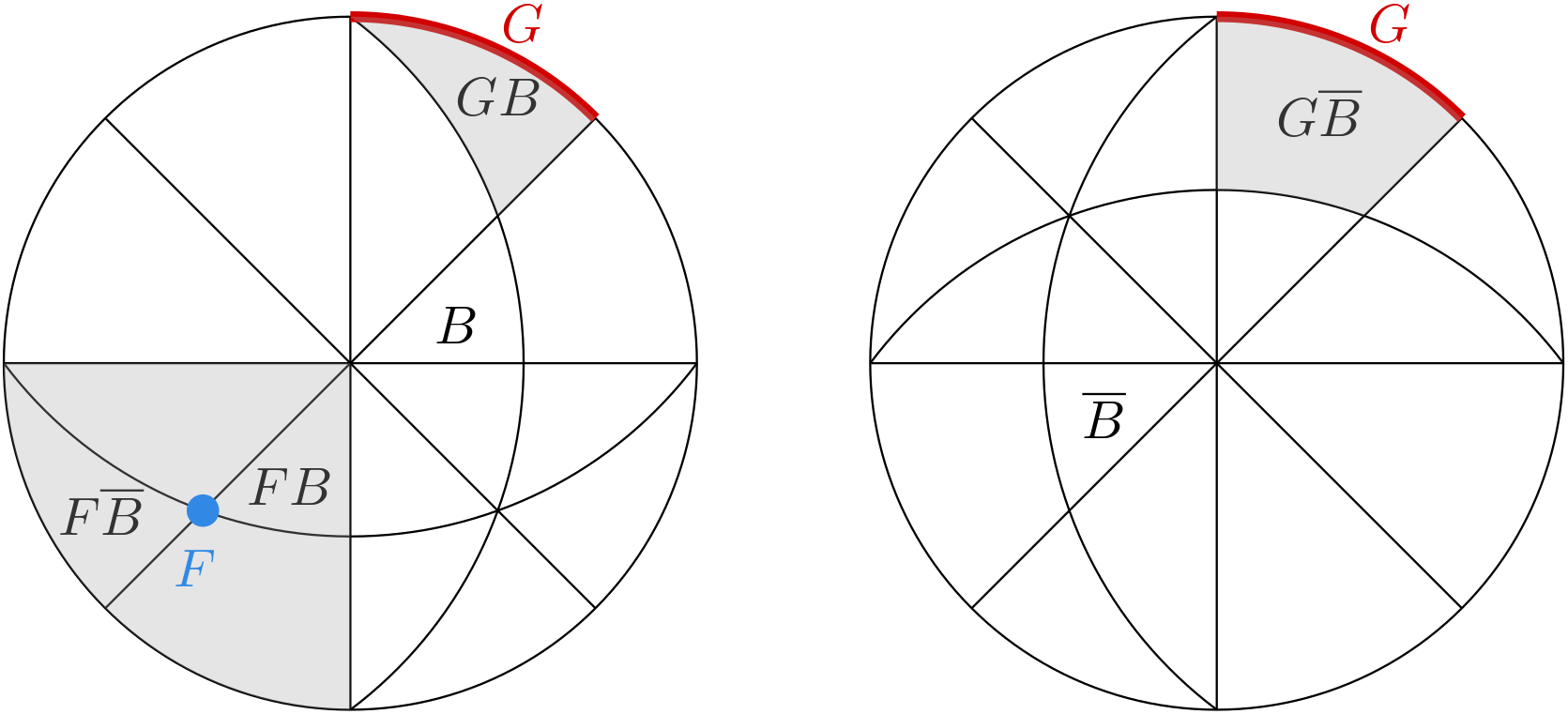}
\caption{A non-simplicial arrangement $\arr$.
         The faces $F$ (blue) and $G$ (red) have rank $1$ and $2$, respectively.
         The top-stars $\cham_F$ and $\cham_G$ (shaded) contain a minimum and maximum element in the order $\preceq_B$.}
\label{f:top-stars}
\end{figure}

\subsection{Descents sets}

Fix a base region $B \in \cham[\arr]$.
Given a polyhedral subcomplex $\Delta \subseteq \face[\arr]$, its \defn{descent set} is
\[
	\Des_B(\Delta) := \{ F \in \face[\arr] \,:\, F \opp{B} \in \Delta \}.
\]
Observe that, since $F \leq FC$ for any faces $F$ and $C$ (see \cref{eq:Tits-properties}), we have that $\Des_B(\Delta) \subseteq \Delta$.
However, $\Des_B(\Delta)$ might fail to be a complex, as illustrated in \cref{f:not-subcomplex}.

\begin{figure}[ht]
\includegraphics[height=.2\textheight]{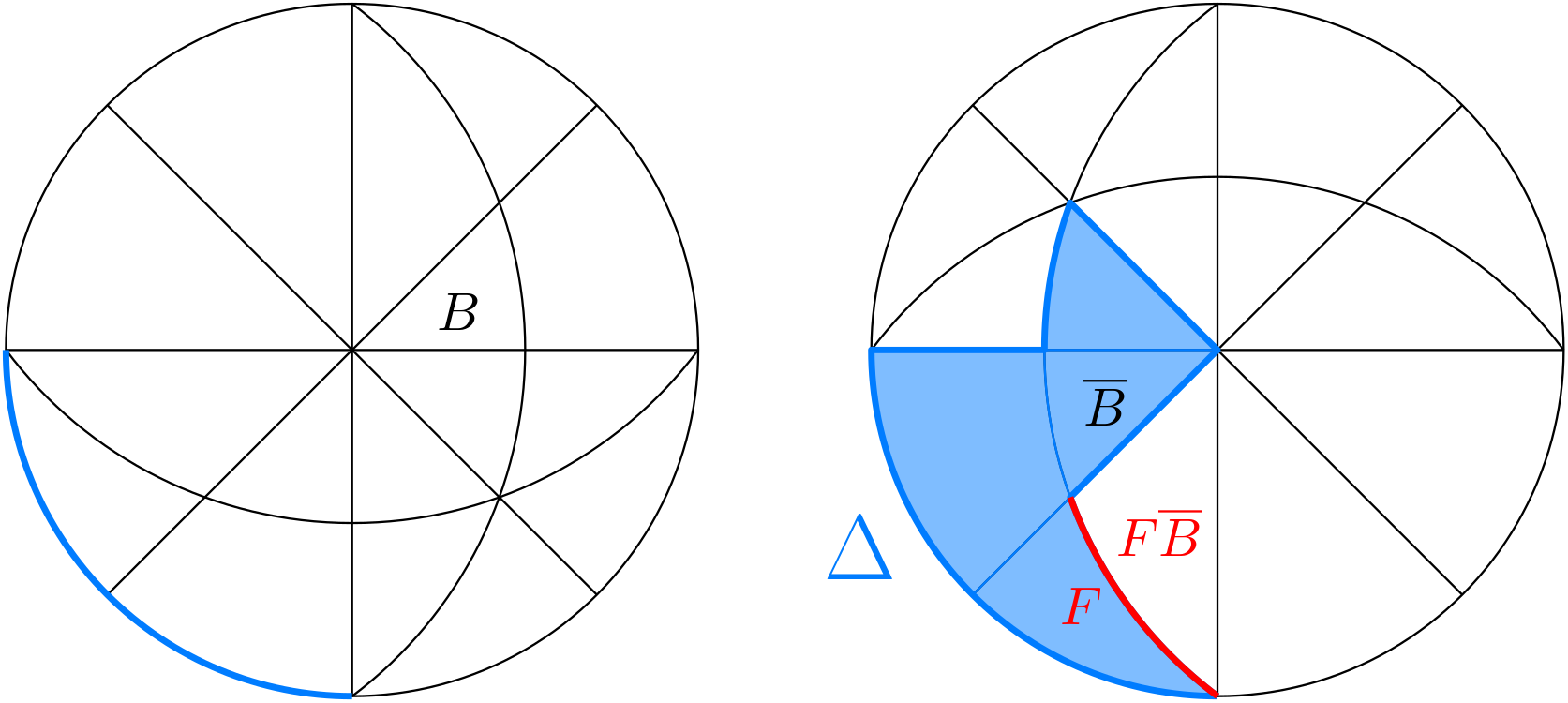}
\caption{A subcomplex $\Delta \subseteq \face[\arr]$ in blue, it includes face $F$ in red.
         Observe, however, that $F\opp{B} \notin \Delta$.
         Thus, $F \notin \Des_B(\Delta)$, $\Des_B(\Delta) \neq \Delta$, and $\Des_B(\Delta)$ is not a complex.}
\label{f:not-subcomplex}
\end{figure}

\begin{lemma}\label{l:subcomplex-is-complex}
Let $\Delta \subseteq \face[\arr]$ be a nonempty polyhedral subcomplex.
Then,
$\Des_B(\Delta) = \Delta$
if and only if $\Delta$ is pure and $\cham(\Delta) := \Delta \cap \cham[\arr]$ is a nonempty upper set of the partial order $\preceq_B$.
\end{lemma}

\begin{proof}
Suppose $\Des_B(\Delta) = \Delta$.
Since $O \in \Delta$, we have $\opp{B} = O \opp{B} \in \Delta$, so $\cham(\Delta)$ is not empty.
Moreover, every face $F \in \Delta$ is contained in an element of $\cham(\Delta)$, namely $F \opp{B}$, so $\Delta$ is pure.
Let $C \in \cham(\Delta)$ and $D \in \cham[\arr]$ be a region covering $C$ in the order $\preceq_B$.
Let $F = C \cap D$ be the common facet of $C$ and $D$.
In particular, $F \leq C$ and $F \in \Delta$.
Since the top-star of $F$ is simply $\cham_F = \{C,D\}$ and $C \preceq_B D$, we have $F \opp{B} = \max_{\preceq_B} \cham_F = D \in \Des_B(\Delta)$.
Since this occurs for every $C \in \cham(\Delta)$ and region $D$ covering $C$, we conclude that $\cham(\Delta)$ is an upper set.

Now assume $\Delta$ is pure and $\cham(\Delta)$ is a nonempty upper set of $\cham$.
Recall that $\Des_B(\Delta) \subseteq \Delta$ holds in general, we prove the reverse inclusion.
Let $F \in \Delta$ and $C \in \cham(\Delta)$ containing $F$, which exists since $\Delta$ is pure.
In particular, $C \in \cham_F$ and $C \preceq_B \max_{\preceq_B} \cham_F = F \opp{B}$.
Since $\cham(\Delta)$ is an upper set, we have that $F \opp{B} \in \Delta$, so $F \in \Des_B(\Delta)$, as we wanted to show.
\end{proof}

\subsection{The simplicial case}

Let $\arr$ be a simplicial arrangement.
Given a base region $B \in \cham[\arr]$, we let $\des_{\preceq_B}(C)$ denote the number of regions covered by $C$ in the weak order $\preceq_B$.
If the base region $B$ is clear from context, we simply write $\des(C)$.
The collection of faces $F \leq C$ such that $F \opp{B} = C$ forms a boolean poset of rank $\des(C)$, see for example \citep[Section 7.1.1]{am17}.
Therefore,
\begin{equation}\label{eq:des-to-codim1}
(z+1)^{\des(C)} = \sum_{F\,:\, F \opp{B} = C} z^{\codim(F)}.
\end{equation}

The following result is obtained by combining \cref{eq:des-to-codim1} and \cref{l:subcomplex-is-complex}.

\begin{proposition}\label{p:des-codim}
Let $\Delta \subseteq \face[\arr]$ be a pure complex such that $\cham(\Delta)$ is a nonempty upper set.
Then,
\[
\sum_{C \in \cham(\Delta)} (z+1)^{\des(C)} = \sum_{F \in \Delta} z^{\codim(F)}.
\]
\end{proposition}

\begin{remark}\label{r:shell}
The preceding result can also be deduced using the fact that any linear extension of the weak order gives a shelling of the corresponding complex.
See for example \citep[Proposition 4.3.2]{blswz} for a result in the language of oriented matroids,
and \citep[Proposition 3.4]{reading05fans} for a generalization to complete polyhedral fans.
\end{remark}

When $\Delta = \face(\sfh)$ for a generic halfspace $\sfh$, \cref{p:des-codim} takes the following form.

\begin{proposition}\label{p:gen-upset-PEul}
Let $\sfh$ be a generic halfspace with respect to $\arr$ such that
$\cham(\sfh)$ is an upper set.
Then,
\[
P_\arr(z) = \sum_{C \subseteq \sfh} z^{\des(C)}.
\]
\end{proposition}

\begin{proof}
Let $\rmH'$ be an affine hyperplane contained in $\sfh$; it is necessarily a parallel translate of the bounding hyperplane of $\sfh$.
The collection $\arr \cap \rmH' := \{ \rmH \cap \rmH' \,:\, \rmH \in \arr\}$ forms an \emph{affine hyperplane arrangement} in ambient space $\rmH'$.
The bounded faces of $\arr \cap \rmH'$ are precisely those of the form $F \cap \rmH'$ for $F \in \face(\sfh) \setminus \{O\}$.
A result of Zaslavsky \citep[Corollary 9.1]{zaslavsky75facing} shows that the bounded complex of a hyperplane arrangement is pure, and consequently $\face(\sfh)$ is pure.

It now follows from \cref{eq:cochar-faces-arr,p:des-codim} that
\[
P_\arr(z)
= (z-1)^d \Psi_\arr(\tfrac{1}{z-1})
= (z-1)^d \sum_{F \subseteq \sfh} (z-1)^{-\dim(F)}
= \sum_{F \subseteq \sfh} (z-1)^{\codim(F)}
= \sum_{C \subseteq \sfh} z^{\des(C)},
\]
as claimed.
\end{proof}

A halfspace $\sfh$ bounded by a hyperplane in $\arr$ (not generic) satisfies that $\cham(\sfh)$ is an upper set if and only if it contains the region $\opp{B}$.
The same is not true for generic halfspaces, as illustrated in \cref{f:generic-no-upset}.
In what follows we describe a class of simplicial arrangements for which a generic \emph{simultaneous} choice of $B$ and $\sfh$ always satisfy the hypothesis of \cref{p:gen-upset-PEul}.

\begin{figure}[ht]
\includegraphics[height=.2\textheight]{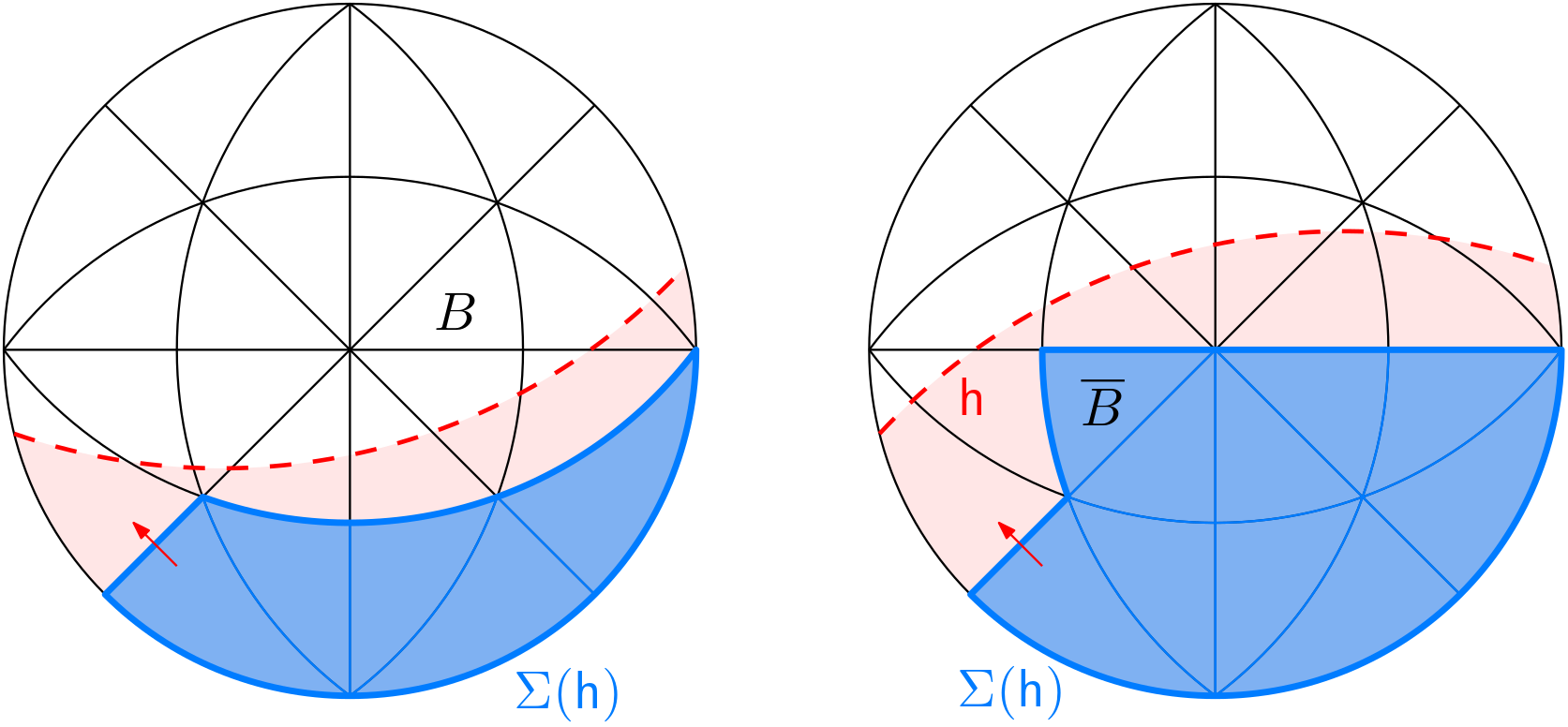}
\caption{Not every halfspace $\sfh$ containing the maximum region $\opp{B}$ satisfies that $\cham(\sfh)$ is an upper set of $\preceq_B$.
			The arrows show two instances of cover relations $C \preceq D$ where $C \in \cham(\sfh)$ yet $D \notin \cham(\sfh)$.}
\label{f:generic-no-upset}
\end{figure}

\subsection{Sharp arrangements}

In this section we use $d$ to denote the dimension of the ambient space, and reserve $n$ to denote normal vectors.

\begin{definition}
We say that an arrangement $\arr$ in $\RR^d$ is \defn{sharp} if it is simplicial, and the angle between any two facets of any region $C \in \cham[\arr]$ is at most $\tfrac{\pi}{2}$.
\end{definition}

Notably, all finite reflection arrangements are sharp.
In particular, the coordinate arrangement in $\RR^d$ is sharp for all $d$.
\cref{f:nonsharp} presents a non-sharp arrangement combinatorially isomorphic to the coordinate arrangement in $\RR^2$.
This shows that sharpness is a geometric and not a combinatorial condition.
Moreover, observe that an arrangement $\arr$ is sharp if and only if its essentialization is sharp.

Assume that $\arr$ is essential and take a region $C \in \cham[\arr]$.
Let $r_1,r_2,\dots,r_d$ be vectors spanning its rays (faces of dimension $1$).
Let $n_1,\dots,n_d$ be the basis of $\RR^d$ dual to $r_1,r_2,\dots,r_d$; that is $\langle n_i , r_j \rangle = \delta_{i,j}$.
It follows that $n_k$ is an inward normal to the facet $F_k \lessdot C$ with rays spanned by $r_1,\dots,r_{k-1},r_{k+1},\dots,r_d$.
Thus, the angle between facets $F_i$ and $F_j$ is $\cos^{-1}(- \tfrac{\langle n_i , n_j \rangle}{|\!|n_i|\!|\,|\!|n_j|\!|})$.
That is, an arrangement $\arr$ is sharp if for all regions $C \in \cham[\arr]$ we have $\langle n_i , n_j \rangle \leq 0$ for $i \neq j$.
In this case, since
\[
n_k =  \langle n_k , n_k \rangle r_k + \sum_{j \neq k} \langle n_k , n_j \rangle r_j,
\]
we have that
\begin{equation}\label{eq:normal-sharp}
	r_k \in \RR_+ \{r_1,\dots,r_{k-1},n_k,r_{k+1},\dots,r_d\},\quad \text{for } k = 1,\dots, d.
\end{equation}

A vector $v \in \RR^d$ is \defn{generic with respect to $\arr$} if $v$ is not contained in any hyperplane of $\arr$.
Given a generic $v \in \RR^d$, let $B(v) \in \cham[\arr]$ be the region of $\arr$ containing $v$, and let $\sfh_v^-$ be the halfspace $\{ x \in \RR^d \,:\, \langle v , x \rangle \leq 0 \}$.
We say that $v$ is \defn{very generic} if in addition $\sfh_v^-$ is generic with respect to $\arr$.

\begin{figure}[ht]
\includegraphics[height=.18\textheight]{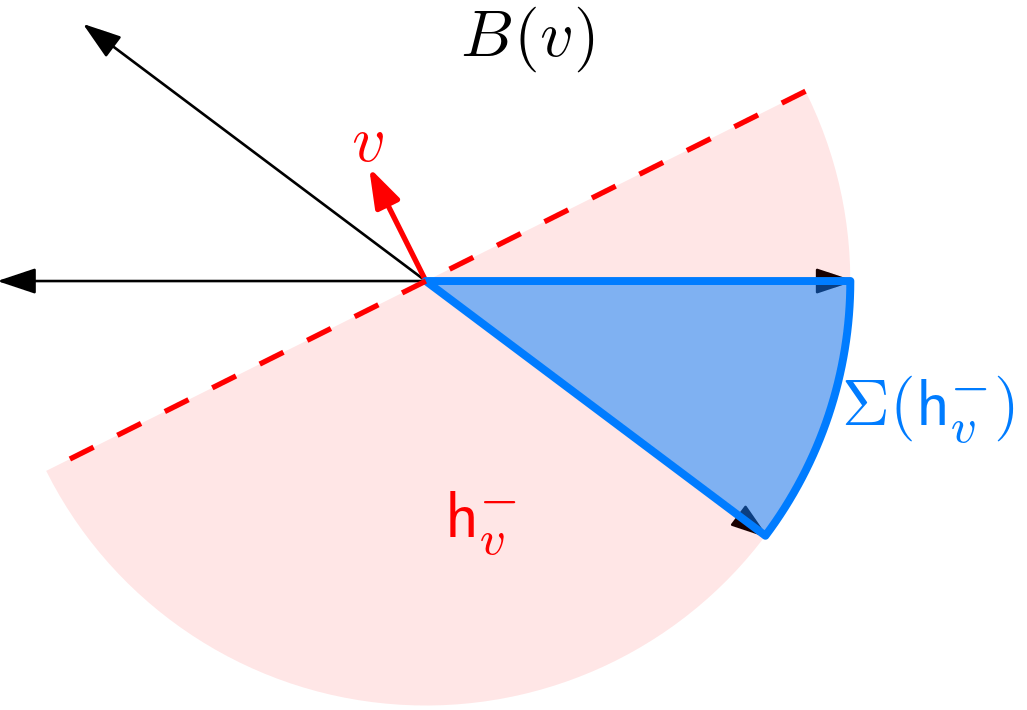}
\caption{A rank 2 arrangement that is \underline{not sharp}.
         For this particular choice of $v$, the collection $\cham(\sfh_v^-)$ (consisting of a single region)
         is not an upper set with respect to $\preceq_{B(v)}$.
         For the region $C$ in $\cham(\sfh_v^-)$ we have $\des(C) = 1$, but $P_{\arr}(z) = z^2$.
         That is, equation \cref{eq:sharp-distr} does not hold for this choice of $v$.}
\label{f:nonsharp}
\end{figure}

\begin{theorem}[Theorem A]\label{t:sharp-distr}
Let $\arr$ be a sharp arrangement.
Then, for any very generic vector $v \in \RR^d$,
\begin{equation}\label{eq:sharp-distr}
P_\arr(z) = \sum_{C \subseteq \sfh_v^-} z^{\des_{\preceq_{B(v)}}(C)}.
\end{equation}
\end{theorem}

\begin{proof}
By \cref{p:gen-upset-PEul}, it is enough to show that $\cham(\sfh_v^-)$ is an upper set of the weak order with base region $B(v)$.
Take regions $C,D$ with $C \in \cham(\sfh_v^-)$ and $D$ covering $C$ in the order $\preceq_{B(v)}$,
we claim that $D \in \cham(\sfh_v^-)$.

By intersecting with the space orthogonal to $\bot$ if necessary, we assume without loss of generality that $\arr$ is essential.
Let $\rmH \in \arr$ be the hyperplane separating $C$ and $D$, and $n_\rmH$ be a normal vector of $\rmH$ such that $C \subseteq H^+ := \{ x \in \RR^d \,:\, \langle n_\rmH , x \rangle \geq 0 \}$.
Since $C \preceq_{B(v)} D$, we have that $B(v)$ is also contained in $H^+$. 
In particular, $\langle n_\rmH , v \rangle > 0$.
Let $r_1,\dots,r_{d-1},r_d$ be the vectors spanning the rays of $D$ numbered so that $r_1,\dots,r_{d-1}$ are the rays of the common facet between $C$ and $D$, namely $C \cap D$.
In particular, $r_1,\dots,r_{d-1}$ are rays of $C$ and, since $C \subseteq \sfh_v^-$ by hypothesis, $\langle r_i , v \rangle < 0$ for all $i \in [d-1]$.
Observe that $n_\rmH = - \lambda n_d$ for some $\lambda > 0$.
It then follows from \cref{eq:normal-sharp} with $k=d$ that $\langle r_d , v \rangle < 0$ and $D \subseteq \sfh^-_v$, as we wanted to show.
\end{proof}

\section{The Primitive Eulerian polynomial of finite Coxeter arrangements}\label{s:PEulCox}
	
Let $(W,S)$ be a finite Coxeter system and $\arr_W$ be the associated reflection arrangement.
For the combinatorics of Coxeter groups, and realizations of the groups of type A, B, and D as permutation groups, we refer the reader to the book of Bj\"{o}rner and Brenti \citep{bb05Coxeter}.

The \defn{Eulerian polynomial} of $\arr_W$ is 
\[
E_W(z) = \sum_{w \in W} z^{\des(w)},
\]
where $\des$ denotes the \defn{descent} statistic on $W$.
That is $\des(w) := \#\{ s \in S \,:\, l(ws) < l(w) \}$ and $l(w)$ denotes the length of $w$ with respect to $S$.

Let us now choose a base region $B \in \cham[\arr_W]$.
The group $W$ acts simply transitively on $\cham[\arr_W]$ and $\des(w) = \des_{\preceq_B}(wB)$ for every $w \in W$.
A real reflection arrangement is always {sharp}, so \cref{t:sharp-distr} yields the following.

\begin{corollary}
For any very generic vector $v$ with respect to $\arr_W$,
\begin{equation}\label{eq:PEul-generic-Cox}
P_W(z) = \sum_{ \substack{w \in W \,:\, wB \subseteq \sfh_v^-} } z^{\des(w)},
\end{equation}
where $B = B(v) \in \cham[\arr_W]$ is the region containing $v$.
\end{corollary}

Since any reflection arrangement is the Cartesian product of irreducible reflection arrangements, we only concentrate in studying the Primitive Eulerian polynomial for irreducible arrangements.
In \cref{ss:PEul-A,ss:PEul-B,ss:PEul-D}, we make an explicit choice of generic $v$ for the arrangements of type A, B, and D, respectively, and use formula \cref{eq:PEul-generic-Cox} to give a combinatorial interpretation for the Primitive Eulerian polynomial of the corresponding type.

The relation between the Eulerian polynomial and the Primitive Eulerian polynomial becomes even more apparent when we compare their generating functions in types A, B, and D.
The generating function for the classical Eulerian polynomials was established by Euler himself \citep{Euler1755}:
\begin{equation}\label{eq:generating-Eul-A}
A(z,x) : = \sum_{n \geq 0} E_{\arr_n}(z) \dfrac{x^n}{n!} = \dfrac{z-1}{z-e^{x(z-1)}}.
\end{equation}
We refer the reader to Foata's survey \citep[Section 3]{foata10eulerian} for a derivation of this formula.
The generating function for the Eulerian polynomials of type B and D are due to Brenti \citep[Theorem 3.4 and Corollary 4.9]{brenti94cox-eul}.
They can be expressed in terms of $A(z,x)$ as follows:
\[
\sum_{n \geq 0} E_{\barr_n}(z) \dfrac{x^n}{n!} = e^{x(z-1)} A(z,2x)
\qquad\qquad
\sum_{n \geq 0} E_{\darr_n}(z) \dfrac{x^n}{n!} = \big( e^{x(z-1)} - z x \big) A(z,2x).
\]
Compare with the generating functions for the Primitive Eulerian polynomials below.

\begin{theorem}\label{t:PEul-generating}
The generating function for the Primitive Eulerian polynomials of type A, B, and D, with $P_{\cD_1}(z) := 0$, are:
\begin{center}
{\renewcommand{\arraystretch}{1.5}
\begin{tabular}{c|c|c}
Type A & Type B & Type D \\
\hline
$1 + \log A(z,x)$ &
$e^{x(z-1)} A(z,2x)^{1/2}$ &
$\big( e^{x(z-1)} - z x \big) A(z,2x)^{1/2}$
\end{tabular}}
\end{center}
\end{theorem}

The formulas in type A and B appear in the proofs of Lemma 5.3 and Lemma 6.2 in \citep{bastidas20polytope}, respectively.
We complete the type D case in \cref{ss:PEul-D}.
Setting $P_{\cD_1}(z) = 0$ is not an accident of the proof, it also simplifies the recursive formulas for the Primitive Eulerian polynomials for type D and related arrangements in \cref{t:PEulD-quad-rec,t:PEul-D-to-B}.

\begin{remark}
The reader might recognize the factor $A(z,2x)^{1/2}$ above:
it is the generating function for the $1/2$-Eulerian polynomials introduced by Savage and Viswanathan \citep{sv12kEulerian}; see \cref{ss:PEul-B} for more details.
\end{remark}

\cref{tab:PrimEulExceptional} shows the Primitive Eulerian polynomial for the exceptional reflection arrangements.
They were computed with the help of SageMath \citep{sagemath}.
The corresponding table for the Eulerian polynomials can be found in Petersen's book \citep[Table 11.5]{petersen15}.

\begin{table}[ht]
{
\renewcommand{\arraystretch}{1.1}
\begin{tabular}{c||l}
$W$ & $P_W(z)$  \\ 
\hline\hline
$H_3$ & $z^3 + 28 z^2 + 16 z$ \\ 
$H_4$ & $z^4 + 1316 z^3 + 3844 z^2 + 900 z$ \\
$F_4$ & $z^4 + 116 z^3 + 220 z^2 + 48 z$ \\ 
$E_6$ & $z^6 + 633 z^5 + 4098 z^4 + 5698 z^3 + 1773 z^2 + 117 z$\\
$E_7$ & $z^7 + 8814 z^6 + 118560 z^5 + 332200 z^4 + 252960 z^3 + 51234 z^2 + 1996 z$ \\
\multirow{2}*{$E_8$} & $ z^8 + 440872z^7 + 11946408z^6 + 60853504z^5$ \\
							& \hspace*{4cm} $+92427088z^4 + 43792992z^3 + 6056496z^2 + 139080z $
\end{tabular}
}
\caption{The Primitive Eulerian polynomial of the reflection arrangements of exceptional type.}
\label{tab:PrimEulExceptional}
\end{table}

\section{The type A Primitive Eulerian polynomial}\label{ss:PEul-A}

The \defn{braid arrangement} $\arr_n$ in $\RR^n$ consists of the hyperplanes with equations $x_i = x_j$ for all $1 \leq i < j \leq n$.
This arrangement is not essential.
Below, we show the arrangements
$\arr_3$ (intersected with the hyperplane perpendicular to $\bot$) and
$\arr_4$ (intersected with the unit sphere inside the hyperplane perpendicular to $\bot$).
\[
\begin{gathered}
\begin{tikzpicture}[<->]
\node [] at (-2.2,1.3) {$\arr_3$};
\draw (-90:1.5) -- (90:1.5) node [above] {\small$x_1 = x_3$};
\draw (150:1.5) -- (-30:1.5) node [below right] {\small$x_1 = x_2$};
\draw (210:1.5) -- (30:1.5) node [above right] {\small$x_2 = x_3$};
\node [right] at (0:.2) {\footnotesize$x_1 \leq x_2 \leq x_3$};
\end{tikzpicture}
\end{gathered}
\hspace*{.2\linewidth}
\begin{gathered}
\begin{tikzpicture}
\node [] at (0,1.3) {$\arr_4$};
\node [] at (0,-1.3) {$\phantom{\arr_4}$};
\end{tikzpicture}
\includegraphics[scale=.12]{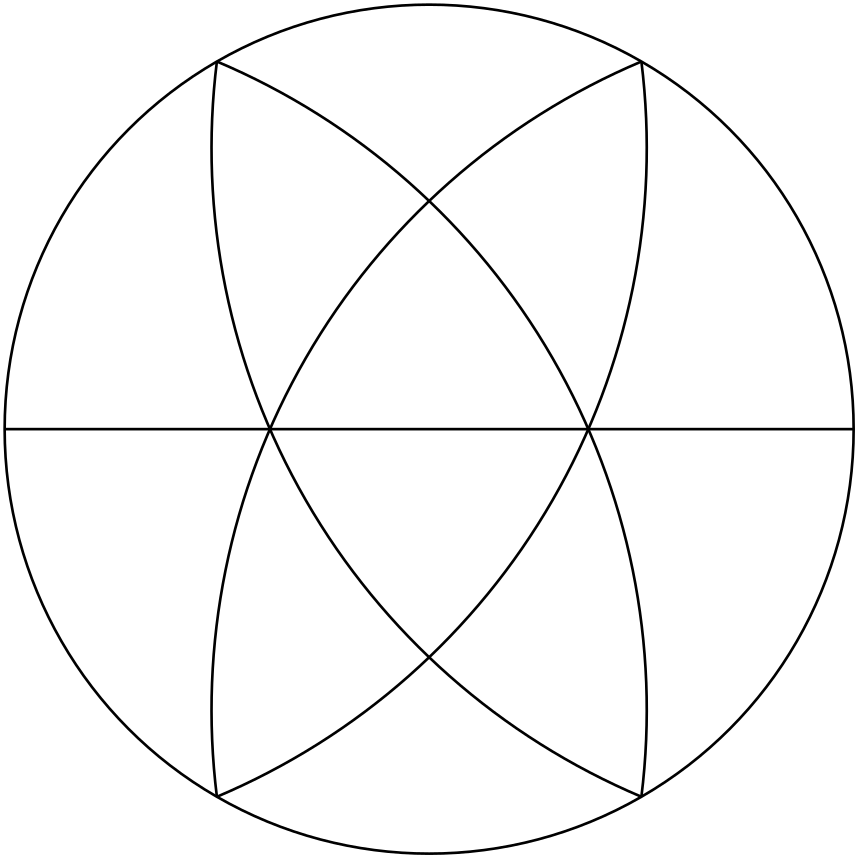}
\end{gathered}
\]
It is the reflection arrangement corresponding to the symmetric group $\frakS_n$, the Coxeter group of type ${\rm A}_{n-1}$.
$\frakS_n$ is the group of permutations $w : [n] \rightarrow [n]$ under composition, where $[n] := \{1,2,\dots, n\}$.
As it is usual, we might write a permutation $w \in \frakS_n$ in its one-line notation $w_1 w_2 \dots w_n$, where $w_i = w(i)$, or as a product of disjoint cycles.
For example, both $2\,4\,1\,3\,5$ and $(1\,2\,4\,3)(5)$ denote the same element of $\frakS_5$.
The \defn{descent} and \defn{excedance} statistic of a permutation $w \in \frakS_n$ are:
\begin{align*}
\des(w) & = \# \{i \in [n-1] \,:\, w(i) > w(i+1) \}; \\
\exc(w) & = \# \{i \in [n-1] \,:\, w(i) > i \}. 
\end{align*}

In \citep[Corollary 5.5]{bastidas20polytope}, the author obtained an interpretation of the coefficients of $P_{\arr_n}(z)$ in terms of the excedance statistic on the \emph{cuspidal} elements of the symmetric group~$\frakS_n$:
\begin{equation}\label{eq:PE-A-exc}
P_{\arr_n}(z) = \sum_{w \in {\rm cusp}(\frakS_n)} z^{\exc(w)}.
\end{equation}
The cuspidal elements of the symmetric group $\frakS_n$ are precisely the \emph{long cycles}; that is, the permutations whose cycle decomposition consists of exactly one cycle of order $n$.

\begin{example}\label{ex:exc-A4}
The distribution of the excedance statistic on the long cycles of $\frakS_4$ is shown below.
\begin{center}
{\renewcommand{\arraystretch}{1.1}
\begin{tabular}{c|c}
$w$ & $\exc(w)$ \\
\hline
$(\red{1} \, \red{2} \, \red{3} \, 4)$ & $3$ \\
$(\red{1} \, 3 \, \red{2} \, 4)$ & $2$ \\
$(2 \, \red{1} \, \red{3} \, 4)$ & $2$
\end{tabular}
\qquad\qquad
\begin{tabular}{c|c}
$w$ & $\exc(w)$ \\
\hline
$(\red{2} \, 3 \, \red{1} \, 4)$ & $2$ \\
$(3 \, \red{1} \, \red{2} \, 4)$ & $2$ \\
$(3 \, 2 \, \red{1} \, 4)$ & $1$
\end{tabular}}
\end{center}
Excedances are marked in red.
Thus, $P_{\arr_4}(z) = z^3 + 4 z^2 + z$.
\end{example}

Let us now interpret the coefficients of $P_{\arr_n}(z)$ using \cref{t:sharp-distr}.
For this, we recall the usual identification between the elements of $\frakS_n$ and the regions of the braid arrangement~$\arr_n$:
\[
w = w_1 w_2 \dots w_n \in \frakS_n
\quad
\longleftrightarrow
\quad
C_w := \{ x \in \RR^n \,:\, x_{w_1} \leq x_{w_2} \leq \dots \leq x_{w_n} \} \in \cham[\arr_n].
\]
Under this identification, the weak order of $\cham[\arr_n]$ with respect to the region $C_e$, where $e \in \frakS_n$ denotes the identity of the symmetric group, coincides with the usual weak order of $\frakS_n$ as a Coxeter group. In particular, $\des(w) = \des_\preceq(C_w)$.

Let $\sfh =\sfh^-_v$ be the halfspace determined by the vector $v = (-1,\dots,-1,n-1) \in \RR^n$.
Bj\"{o}rner and Wachs \citep{bw04geombases} characterized the permutations $w \in \frakS_n$ such that $C_w \subseteq \sfh$,
we let $BW^A_n$ denote the collection of such permutations.
They show that
\[
BW^A_n = \{ w \in \frakS_n \,:\, w_1 = n \}.
\]
The halfspace $\sfh^-_v$ is generic, however $v$ itself is not generic since it lies in the boundary of the region $C_e$.
Nonetheless, we can perturb $v$ so that it lies in the interior of $C_e$ without changing which regions are contained in $\sfh^-_v$.
The following is then a direct consequence of \cref{t:sharp-distr}.

\begin{theorem}\label{t:PE-A-des}
For every $n \geq 1$,
\begin{equation}\label{eq:PE-A-des}
P_{\arr_n}(z) = \sum_{w \in \frakS_n \,:\, w_1 = n} z^{\des(w)}.
\end{equation}
\end{theorem}

\begin{example}
The distribution of the descent statistic on $BW^A_4$ is shown below.
\begin{center}
{\renewcommand{\arraystretch}{1.1}
\begin{tabular}{c|c}
$w$ & $\des(w)$ \\
\hline
$\red{4} \, \red{3} \, \red{2} \, 1$ & $3$ \\
$\red{4} \, 2 \, \red{3} \, 1$ & $2$ \\
$\red{4} \, \red{3} \, 1 \, 2$ & $2$
\end{tabular}
\qquad\qquad
\begin{tabular}{c|c}
$w$ & $\des(w)$ \\
\hline
$\red{4} \, 1 \, \red{3} \, 2$ & $2$ \\
$\red{4} \, \red{2} \, 1 \, 3$ & $2$ \\
$\red{4} \, 1 \, 2 \, 3$ & $1$
\end{tabular}}
\end{center}
The positions where a descent occurs are are in red.
Observe that this distribution agrees with \cref{ex:exc-A4}.
\end{example}

\begin{figure}[ht]
\includegraphics[height=.2\textheight]{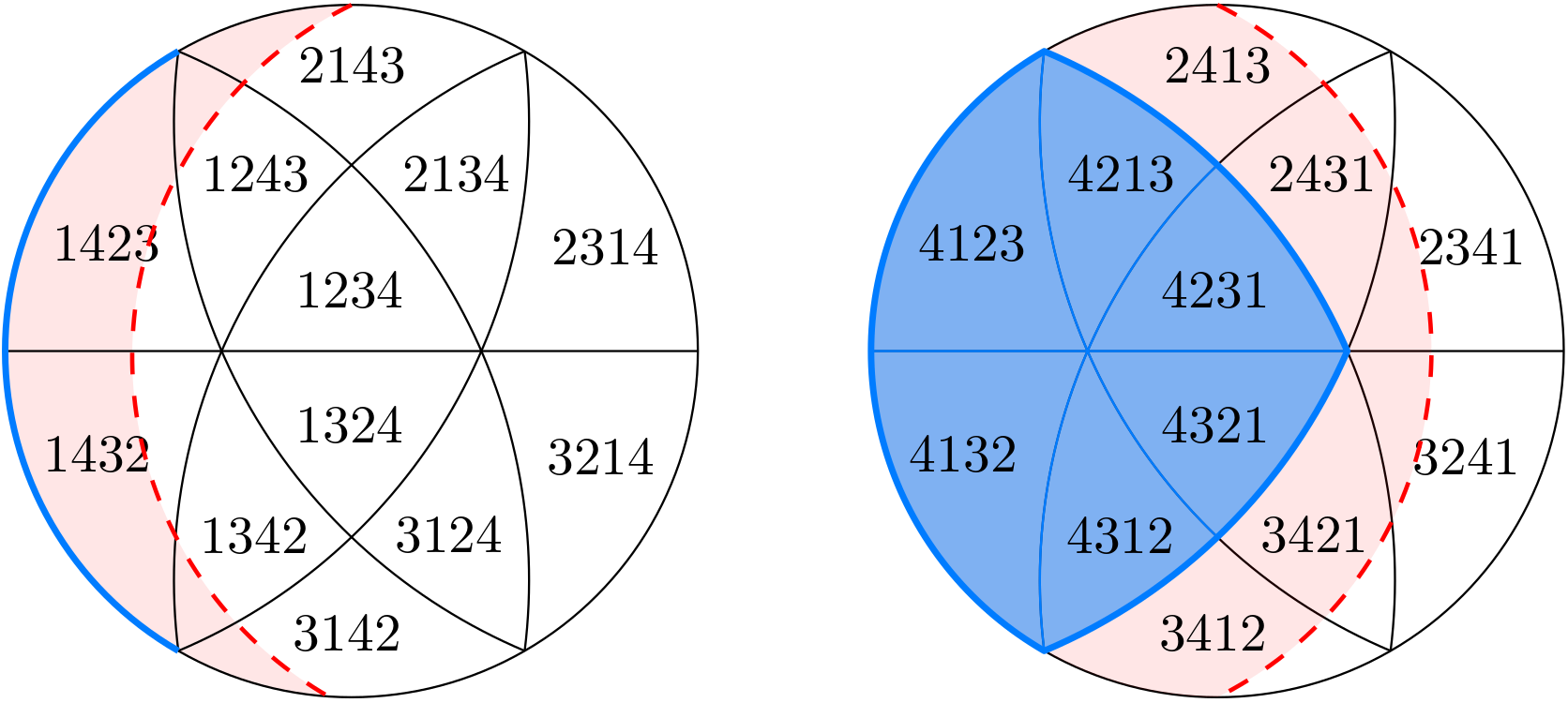}
\caption{Regions corresponding to elements in $BW^A_4$.}
\label{f:BWA4}
\end{figure}

The equivalence between formulas \cref{eq:PE-A-exc,eq:PE-A-des} can be proved combinatorially via Foata's \emph{first fundamental transformation} \citep{foata65}.

There is a natural bijection $BW^A_n \to \frakS_{n-1}$ sending $w = n \, w_2 \dots w_n$ to $w_2 \dots w_n$, in one-line notation.
If $n \geq 2$, this map reduces the descent statistic by exactly one. Therefore, for $n \geq 2$,
\begin{equation}\label{eq:typeA-Eul-is-PEul}
	P_{\arr_n}(z) = z E_{\arr_{n-1}}(z).
\end{equation}
That is, the Primitive Eulerian polynomials of type A are just the usual Eulerian polynomials (with an additional factor of $z$ and a shift of $1$ on its index).
This explains the appearance of the Eulerian numbers in \cref{ex:exc-A4}.
With this identification in mind, the recursion obtained by \cref{p:PEul-recursion} is equivalent to the following well-known quadratic recurrence for the classical Eulerian polynomials:
\[
	E_{\arr_n}(z) = (1+z)E_{\arr_{n-1}}(z) + z \sum_{k=1}^{n-2} \binom{n-1}{k}  E_{\arr_k}(z) E_{\arr_{n-1-k}}(z)
	\qquad
	\text{for all } n \geq 2.
\]
See for example \citep[Theorem 1.6]{petersen15}.

\section{The type B Primitive Eulerian polynomial}\label{ss:PEul-B}

The \defn{type B Coxeter arrangement} $\barr_n$ in $\RR^n$ consists of the hyperplanes with equations $x_i = x_j$, $x_i = - x_j$ for all $1 \leq i < j \leq n$, and $x_i = 0$ for all $1 \leq i \leq n$.
\[
\begin{gathered}
\begin{tikzpicture}[<->]
\node [] at (-2.2,1.3) {$\barr_2$};
\draw (-90:1.5) -- (90:1.5) node [above] {\small$x_1 = 0$};
\draw (180:1.5) -- (0:1.5) node [right] {\small$x_2 = 0$};
\draw (45+180:1.5) -- (45:1.5) node [above right] {\small$x_1 = x_2$};
\draw (-45+180:1.5) -- (-45:1.5) node [below right] {\small$x_1 = -x_2$};
\end{tikzpicture}
\end{gathered}
\hspace*{.2\linewidth}
\begin{gathered}
\begin{tikzpicture}
\node [] at (0,1.3) {$\barr_3$};
\node [] at (0,-1.3) {$\phantom{\barr_3}$};
\end{tikzpicture}
\includegraphics[scale=.12]{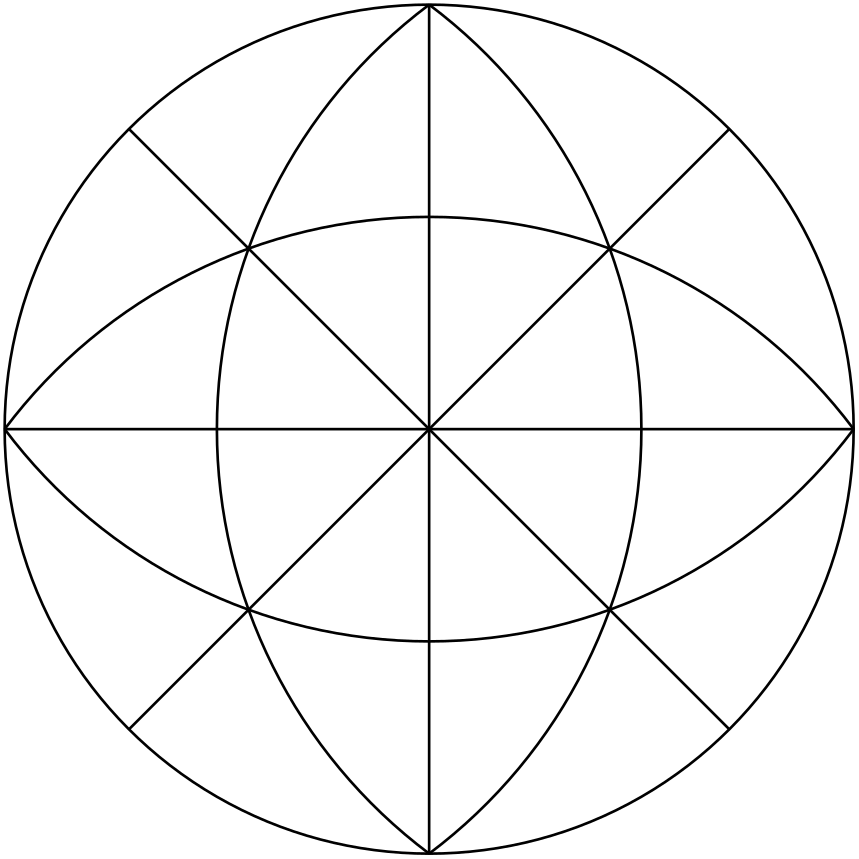}
\end{gathered}
\]
It is the reflection arrangement corresponding to the \defn{hyperoctahedral group} $\frakB_n$, the Coxeter group of type ${\rm B}_n$.
$\frakB_n$ is the group of permutations $w$ of $[\pm n] := \{-n,\dots,-1,1,\dots,n \}$ satisfying $w(-i) = -w(i)$.
Elements of $\frakB_n$ are called \defn{signed permutations}.
The \defn{window notation} of a signed permutation $w \in \frakB_n$ is the word $w_1 w_2 \dots w_n$, where $w_i = w(i)$.
We often write $\opp{j}$ instead of $-j$ for $j \in [\pm n]$, so $\opp{\opp{j}} = j$.

In a signed permutation, cycles can be of two forms:
\begin{align*}
(( i_1 \, i_2 \dots i_k )) := & ( i_1 \, i_2 \dots i_k ) ( \opp{i_1} \, \opp{i_2} \dots \opp{i_k} ), \\
[ i_1 \, i_2 \dots i_k ] := & ( i_1 \, i_2 \dots i_k  \opp{i_1} \, \opp{i_2} \dots \opp{i_k} ),
\end{align*}
where $\{i_1,i_2,\dots,i_k\} \subseteq [\pm n]$ is an \defn{involution-exclusive} subset; see for instance \citep{Kerber71rep1,bw02artin}.
Recall that a subset $S \subseteq [\pm n]$ is involution-exclusive if $S \cap \opp{S} = \emptyset$, where $\opp{S} := \{ \opp{j} \,:\, j \in S\}$.
Cycles of the form $((\dots))$ are called \defn{paired cycles}, while those of the form $[\dots]$ are called \defn{balanced cycles}.
Every element in $\frakB_n$ decomposes uniquely as a product of disjoint paired or balanced cycles, up to reordering and equivalence of cycles. For instance, the following are two expressions in cycle notation for the element $w \in \frakB_7$ whose window notation is
$\opp{1} \, 2 \, \opp{4} \, \opp{3} \, \opp{6} \, \opp{7} \, \opp{5}$:
\[
[1] \, ((2)) \, ((3 \, \opp{4})) \, [5 \, \opp{6} \, 7 ]
=
[\opp{1}] \, ((2)) \, ((\opp{3} \, 4)) \, [\opp{6} \, 7 \, \opp{5}].
\] 
We will make use of the following statistics on signed permutations.
For $w \in \frakB_n$, define
\begin{align*}
\des_A(w) & = \# \{ i \in [n-1] \,:\, w(i) > w(i+1) \};\\
\des_B(w) & = \# \{ i \in [0,n-1] \,:\, w(i) > w(i+1) \}, & \text{where } w(0) := 0 ; \\
\exc_A(w) & = \# \{i \in [n-1] \,:\, w(i) > i \}; \\
\fneg(w) & = \# \{i \in [n] \,:\, w(i) < 0 \}.
\end{align*}
$\des_B$ corresponds to the Coxeter descent statistic, so we will abbreviate it to $\des$ when no confusion arises.

Adin and Roichman \citep{ar01fmaj} pioneered the study of \emph{flag statistics} on the hyperoctahedral group. Soon after, Adin, Brenti, and Roichman \citep{abr01desmaj} introduced the \defn{flag-descent} statistic $\fdes$, which in some sense \emph{refines} $\des$.
Later, Bagno and Garber \citep{bg06colorexc} introduced the \defn{flag-excedance} statistic $\fexc$ (although with a different name and in the more general context of colored permutations).
We review these definitions below:
\begin{align*}
\fdes(w) & = \des_A(w) + \des_B(w); \\
\fexc(w) & = 2 \exc_A(w) + \fneg(w); \\
\exc_B(w) & = \lfloor \dfrac{ {\fexc}(w) + 1 }{2} \rfloor.
\end{align*}
Foata and Han \citep{fh09signedV} proved that $\fexc$ and $\fdes$ have the same distribution.
Since $\des(w) = \des_B(w) = \lfloor \dfrac{ {\fdes}(w) + 1 }{2} \rfloor$ for all signed permutations $w \in \frakB_n$, the statistic $\exc_B$ is also Eulerian (i.e. it has the same distribution as $\des$).

Similar to the type A case, \citep[Corollary 5.5]{bastidas20polytope} gives an interpretation of the coefficients of $P_{\barr_n}(z)$ in terms of the statistic $\exc_B$ on the cuspidal elements of $\frakB_n$:
\begin{equation}\label{eq:PE-B-exc}	
P_{\barr_n}(z) = \sum_{w \in {\rm cusp}(\frakB_n)} z^{\exc_B(w)},
\end{equation}
The cuspidal elements of the hyperoctahedral group $\frakB_n$ are those whose cycle decomposition only involves balanced cycles.

\begin{example}\label{ex:fexc-B3}
The distribution of $\exc_A$, $\fneg$, and $\exc_B$ on ${\rm cusp}(\frakB_3)$ is shown below.
\begin{center}
{
\renewcommand{\arraystretch}{1.1}
\begin{tabular}{c|c|c|c}
$w$ & $\exc_A$ & $\fneg$ & $\exc_B$ \\
\hline
$[1 \, 2 \, 3]$  & $2$ & $1$ & $3$ \\
$[1][2 \, 3]  $  & $1$ & $1$ & $2$ \\
$[2][1 \, 3]  $  & $1$ & $1$ & $2$ \\
$[3][1 \, 2]  $  & $1$ & $1$ & $2$ \\
$[1 \, 3 \, 2]$  & $1$ & $1$ & $2$
\end{tabular}
\quad
\begin{tabular}{c|c|c|c}
$w$ & $\exc_A$ & $\fneg$ & $\exc_B$ \\
\hline
$[3 \, 1 \, 2]      $ & $1$ & $1$ & $2$ \\
$[2 \, 1 \, 3]      $ & $1$ & $1$ & $2$ \\
$[2 \, 3 \, 1]      $ & $1$ & $1$ & $2$ \\
$[1 \, \opp{3} \, 2]$ & $0$ & $3$ & $2$ \\
$[1 \, \opp{2} \, 3]$ & $0$ & $3$ & $2$
\end{tabular}
\quad
\begin{tabular}{c|c|c|c}
$w$ & $\exc_A$ & $\fneg$ & $\exc_B$ \\
\hline
$[1][2][3]    $ & $0$ & $3$ & $2$ \\
$[1][3 \, 2]  $ & $0$ & $2$ & $1$ \\
$[2][3 \, 1]  $ & $0$ & $2$ & $1$ \\
$[3][2 \, 1]  $ & $0$ & $2$ & $1$ \\
$[3 \, 2 \, 1]$ & $0$ & $1$ & $1$ 
\end{tabular}
}
\end{center}
Thus, $P_{\barr_3}(z) = z^3 + 10 z^2 + 4 z$.
\end{example}

\begin{remark}
Other definitions of \emph{excedance-like} statistics on the hyperoctahedral group
have been considered, in chronological order, by
Steingrimsson \citep{steingrimsson92},
Brenti \citep{brenti94cox-eul},
Fire \citep{fire04statistics},
Bagno and Garber~\citep{bg06colorexc}.
In fact, except for Brenti's definition, these apply in the more general context of colored permutations.
Restricted to the hyperoctahedral group:
\begin{itemize}[wide]
\item Steingrimsson and Brenti's statistics are both Eulerian and have the same distribution as $\exc_B$ on cuspidal elements.
      That is, Formula \cref{eq:PE-B-exc} can also be expressed in terms of these statistics.

\item Fire's statistic (which agrees with Bagno and Garber absolute excedance) has the same distribution as $\fexc$, but not when restricted to cuspidal elements.

\item The color excedance of Bagno and Garber coincides with $\fexc$.
\end{itemize}
\end{remark}

Let us now interpret the coefficients of $P_{\barr_n}(z)$ using \cref{t:sharp-distr}.
We identify the elements of $\frakB_n$ and the regions of the type B Coxeter arrangement $\barr_n$ as follows:
\[
w = w_1 w_2 \dots w_n \in \frakB_n
\quad
\longleftrightarrow
\quad
C_w := \{ x \in \RR^n \,:\, 0 \leq x_{w_1} \leq x_{w_2} \leq \dots \leq x_{w_n} \} \in \cham[\barr_n],
\]
where $x_{\opp{i}} = - x_i$ for all $i \in [n]$.
Again, the weak order of $\cham[\barr_n]$ with respect to the region $C_e$ coincides with the weak order of $\frakB_n$ as a Coxeter group. In particular, $\des_B(w) = \des_\preceq(C_w)$.

Let $\sfh = \sfh_v^-$ be the halfspace determined by the vector $v = (1,2,4,\dots,2^{n-1}) \in \RR^n$.
Then, $C_w \subseteq \sfh$ if and only if all the right-to-left maxima of $|w| := |w_1| |w_2| \dots |w_n| \in \frakS_n$ are negative in $w$, see \citep[Proposition 7.2]{bw04geombases}.
We let $BW^B_n$ denote the collection of such elements.
The vector $v$ is very generic and $v \in C_e$.
The following is a direct consequence of \cref{t:sharp-distr}.

\begin{theorem}\label{t:PE-B-des}
For all $n \geq 1$,
\[
	P_{\barr_n}(z) = \sum_{w \in BW^B_n} z^{\des(w)}.
\]
\end{theorem}

\begin{example}
The distribution of the descent statistic on $BW^B_3$ is shown below.
\begin{center}
{\renewcommand{\arraystretch}{1.1}
\begin{tabular}{c|c}
$w$ & $\des(w)$ \\
\hline
$\opp{1} \, \opp{2} \, \opp{3}$ & $3$ \\
$      1 \, \opp{2} \, \opp{3}$ & $2$ \\
$\opp{1} \,       2 \, \opp{3}$ & $2$ \\
$\opp{2} \, \opp{1} \, \opp{3}$ & $2$ \\
$      2 \, \opp{1} \, \opp{3}$ & $2$ 
\end{tabular}
\qquad\qquad
\begin{tabular}{c|c}
$w$ & $\des(w)$ \\
\hline
$\opp{2} \, \opp{1} \, \opp{3}$ & $2$ \\
$      2 \,       1 \, \opp{3}$ & $2$ \\
$\opp{1} \, \opp{3} \, \opp{2}$ & $2$ \\
$\opp{3} \, \opp{1} \, \opp{2}$ & $2$ \\
$\opp{3} \,       1 \, \opp{2}$ & $2$ 
\end{tabular}
\qquad\qquad
\begin{tabular}{c|c}
$w$ & $\des(w)$ \\
\hline
$\opp{2} \, \opp{3} \, \opp{1}$ & $2$ \\
$      1 \,       2 \, \opp{3}$ & $1$ \\
$      1 \, \opp{3} \, \opp{2}$ & $1$ \\
$      2 \, \opp{3} \, \opp{1}$ & $1$ \\
$\opp{3} \, \opp{2} \, \opp{1}$ & $1$ 
\end{tabular}}
\end{center}
Observe that this distribution agrees with \cref{ex:fexc-B3}.
\end{example}

\begin{figure}[ht]
\includegraphics[height=.2\textheight]{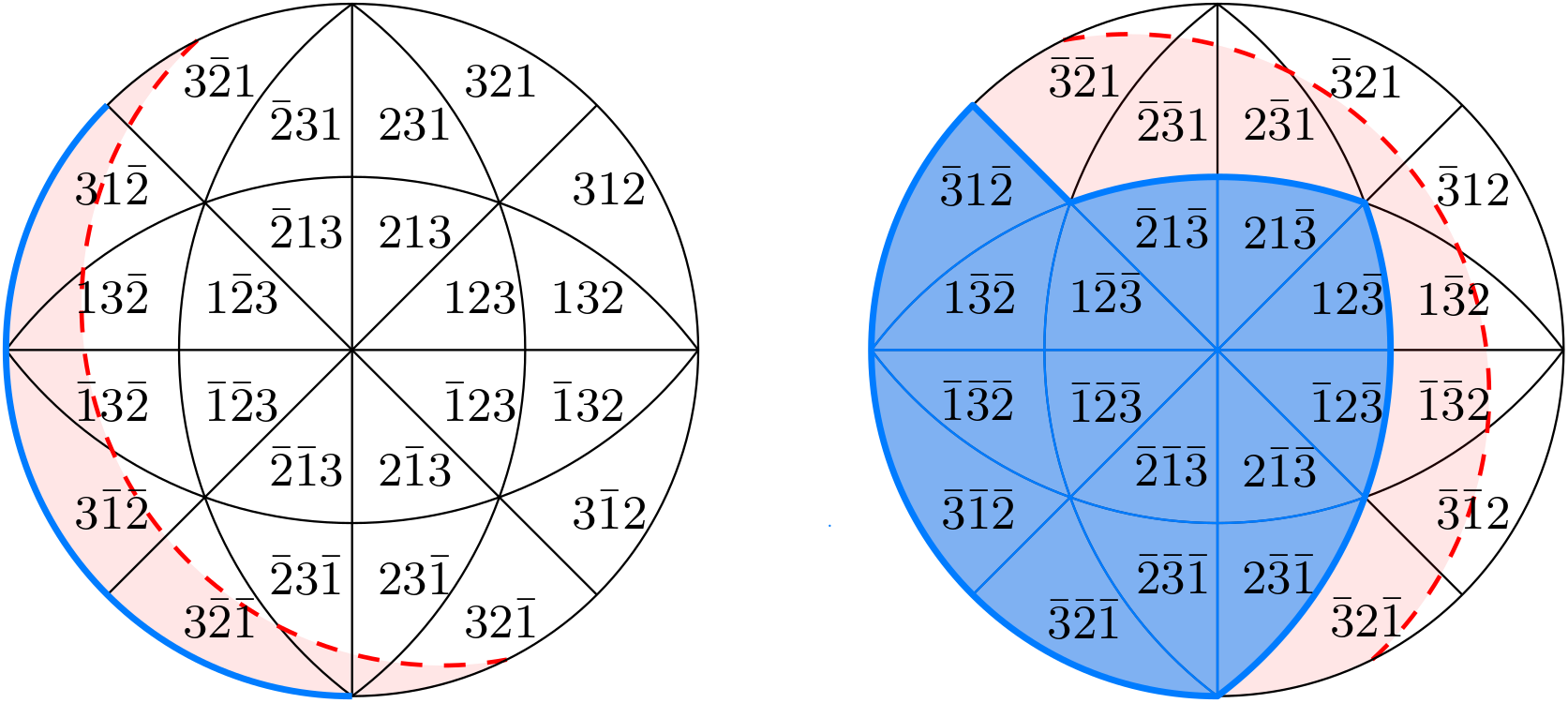}
\caption{Regions corresponding to elements in $BW^B_3$.}
\end{figure}

\cref{tab:PEul-B} shows the Primitive Eulerian polynomials of type B for the first values of $n$.

\begin{table}[ht]
\begin{tabular}{c||l}
$n$ & $P_{\barr_n}(z)$ \\
\hline\hline
$0$ & $1$ \\
$1$ & $z$ \\
$2$ & $z^2 + 2 z$ \\
$3$ & $z^3 + 10 z^2 + 4 z$ \\
$4$ & $z^4 + 36 z^3 + 60 z^2 + 8 z$ \\
$5$ & $z^5 + 116 z^4 + 516 z^3 + 296 z^2 + 16 z$
\end{tabular}
\caption{The first Primitive Eulerian polynomials of type B.
         The coefficients form the sequence \oeis{A185411}.}
\label{tab:PEul-B}
\end{table}

\subsection{Relation to the 1/2-Eulerian polynomials}

The coefficients appearing in \cref{tab:PEul-B} are the \defn{$1/2$-Eulerian numbers}, as introduced by Savage and Viswanathan in \citep{sv12kEulerian}.
However, the order of the coefficients is reversed with respect to the {$1/2$-Eulerian polynomials}.
We prove this observation in \cref{p:PEulB-1_2Eul}.

Let $I_{n,2}$ be the set of \defn{$2$-inversion sequences} defined by
\[
I_{n,2} = \{ e \in \ZZ^n \,:\, 0 \leq e_i \leq 2(i-1) \}.
\]
For $e \in I_{n,2}$, define
\[
\asc(e) = \# \bigg\{ i \in [n-1] \,:\, \dfrac{e_i}{2(i-1) + 1} < \dfrac{e_{i+1}}{2i + 1} \bigg \}.
\]
The \defn{$1/2$-Eulerian polynomial} $A_n^{(2)}(z)$ is the polynomial that keeps track of the distribution of $\asc$ on $I_{n,2}$.
Explicitly,
\[
A_n^{(2)}(z) := \sum_{e \in I_{n,2}} z^{\asc(e)}.
\]
Savage and Viswanathan \citep{sv12kEulerian} showed that the exponential generating function of the polynomials $A_n^{(2)}(z)$ is $A(z,2x)^{1/2}$, with $A(z,x)$ as in equation \cref{eq:generating-Eul-A}.
A more general definition of $1/k$-Eulerian polynomial $A^{(k)}_n(z)$ exists.
We do not explore these polynomials here.

\begin{proposition}\label{p:PEulB-1_2Eul}
For all $n \geq 0$, $P_{\barr_n}(z) = z^n A^{(2)}_n(\tfrac{1}{z})$.
\end{proposition}

\begin{proof}
In view of Savage and Viswanathan's result, the generating function of the polynomials $z^n A^{(2)}_n(\tfrac{1}{z})$ is
\[
A(\tfrac{1}{z},2xz)^{1/2} =
\Big( \dfrac{\tfrac{1}{z}-1}{\tfrac{1}{z}-e^{2zx(\tfrac{1}{z}-1)}}  \Big)^{1/2} = 
\Big( \dfrac{(1-z) e^{2x(z-1)}}{e^{2x(z-1)}-z} \Big)^{1/2} = 
e^{x(z-1)} A(z,2x)^{1/2}. 
\]
This is precisely the generating function of the polynomials $P_{\barr_n}(z)$ in \cref{t:PEul-generating}.
\end{proof}

There are yet more ways to interpret the polynomials $A^{(k)}_n(z)$.
Savage and Schuster \citep{ss12lecturehall} showed that $A^{(k)}_n(z)$ is the $h^*$-polynomial of the \emph{$k$-lecture hall polytope} of dimension $n$; and Ma and Mansour \citep{mm15kStirling} proved that $A^{(k)}_n(z)$ is the generating function for the \emph{ascent-plateau} statistic on \emph{$k$-Stirling permutations}.
More recently, Tolosa Villareal \citep{tolosa21thesis} conjectured that $A^{(2)}_n(z)$ is the $h$-polynomial of the \defn{positive signed sum system} $\Delta(\cQ^+(v))$ associated to a \emph{generic weight}~$v$.
We recall the definition of $\Delta(\cQ^+(v))$ and prove this conjecture below.

A weight $v \in \RR^n$ is \defn{generic} if for all involution-exclusive subsets $\emptyset \subsetneq J \subseteq [\pm n]$, $v_J := \sum_{i \in J} v_i \neq 0$.
To such a weight $v$, we associate the collection $\cQ^+(v) = \{ J \,:\, v_J > 0 \}$.
Elements of $\cQ^+(v)$ are partially ordered by inclusion, and $\Delta(\cQ^+(v))$ denotes the corresponding chain complex.

\begin{proposition}[{\citep[Conjecture 5.1]{tolosa21thesis}}]\label{p:Tolosa}
For every generic weight $v \in \RR^n$, the $h$-polynomial of $\Delta(\cQ^+(v))$
is $A^{(2)}_n(z)$.
\end{proposition}

\begin{proof}
Proper involution-exclusive subsets correspond to rays of $\barr_n$.
In this manner, $\cQ^+(v)$ corresponds to the rays of $\barr_n$ contained in the half-space $\sfh = \{ x \in \RR^n \,:\, \langle v , x \rangle \geq 0 \}$, and $\Delta(\cQ^+(v))$ corresponds to the complex $\face(\sfh)$.
Since the $f$-polynomial, and therefore the $h$-polynomial, of $\Delta(\cQ^+(v))$ is independent of $v$, we can assume that $-v$ lies in the fundamental region of $\frakB_n$.
The result follows by \cref{p:des-codim,p:PEulB-1_2Eul}.
\end{proof}

\begin{remark}
Tolosa Villareal's conjecture is originally stated in terms of the $h^*$-polynomial of the $2$-lecture hall polytope.
The statement above is equivalent in view of Savage and Schuster's work \citep[Theorem 5]{ss12lecturehall}.
\end{remark}

Tolosa Villareal also conjectured a stronger version of \cref{p:Tolosa}: The $2$-lecture hall polytope $\cP_{n,2}$ has a unimodular triangulation that is combinatorially isomorphic to the cone over $\Delta(\cQ^+(1,2,\dots,2^n))$ \citep[Conjecture 5.2]{tolosa21thesis}. Constructing such triangulations for all $n$ is a work in progress of Ardila and Tolosa Villareal (personal communication).

In view of \cref{p:PEulB-1_2Eul}, all the recurrences for the polynomials $A^{(2)}(z)$ discovered by Savage and Viswanathan can be translated into recurrences for the type B Primitive Eulerian polynomials.
For instance, we can verify that the generating function $P_\cB(z,x)$ for the polynomials $P_{\cB_n}(z)$ satisfies the differential equation
\[
z P_\cB(z,x) + 2(x z - 1) \dfrac{\partial P_\cB(z,x)}{\partial x} + 2z(1 - z)\dfrac{\partial P_\cB(z,x)}{\partial z} = 0,
\]
and deduce the following result, which is a specialization of the recurrence (20) in \citep{sv12kEulerian}.

\begin{proposition}The type B Primitive Eulerian polynomials are determined by the differential recurrence
\[
	P_{\barr_n}(z) = (2n-1)zP_{\barr_{n-1}}(z) + 2z(1-z)P'_{\barr_{n-1}}(z),
\]
with initial condition $P_{\barr_{0}}(z) = 1$.
\end{proposition}

This is analogous to the following recurrence for the type B Eulerian polynomials:
\[
E_{\barr_n}(z) = (1+(2n-1)z)E_{\barr_{n-1}}(z) + 2z(1-z)E'_{\barr_{n-1}}(z).
\]

\subsection{Some recurrences}

Remarkably, \cref{p:PEul-recursion} gives a new recurrence for the polynomials $P_{\barr_n}(z)$ that does not follow from those known for $A_n^{(2)}(z)$.

\begin{theorem}\label{t:PEulB-recursion}
The type B Primitive Eulerian polynomials satisfy the following recursion.
With $P_{\barr_0}(z) = 1$,
\[
P_{\barr_n}(z) = zP_{\barr_{n-1}}(z) + \sum_{k=1}^{n-1} \binom{n-1}{k} 2^k P_{\barr_{n-1-k}}(z)P_{\arr_{k+1}}(z),
\]
for all $n \geq 1$.
\end{theorem}

This is analogous to the following quadratic recurrence for the type B Eulerian polynomials
\[
E_{\barr_n}(z) = (1+z)E_{\barr_{n-1}}(z) + 2z\displaystyle\sum_{k=1}^{n-1} \binom{n-1}{k} 2^k E_{\barr_{n-1-k}}(z)E_{\arr_{k}}(z),
\]
with $E_{\barr_0}(z) = 1$. See for example \citep[Theorem 13.2]{petersen15}.

In order to prove \cref{t:PEulB-recursion}, we need to review the correspondence between flats of $\barr_n$ and type B partitions of $[\pm n]$.

A \defn{type B set partition} of $[\pm n]$ is a (weak) set partition $\rmX = \{S_0 , S_1 , \opp{S_1} ,\dots,S_k,\opp{S_k}\}$ of $[\pm n]$ such that $S_0 = \opp{S_0}$ is the only block allowed to be empty.
The block $S_0$ is called the \defn{zero block} of $\rmX$.
Since blocks are pairwise disjoint, the nonzero blocks are involution-exclusive.
We write~$\rmX \vdash^B [\pm n]$ to denote that~$\rmX$ is a type B set partition of~$[\pm n]$.
Given a partition~$\rmX \vdash^B [\pm n]$, the corresponding flat of $\barr_n$ is the intersection of the hyperplanes~$x_i = x_j$ for all~$i,j$ that belong to the same block of $\rmX$ (recall that $x_{-i} = -x_i$ for $i \in [n]$).
Observe that $\dim(\rmX)$ is half the number of nonzero blocks of $\rmX$ as a type B partition.
The partial order relation of $\flat[\barr_n]$ becomes the ordering by refinement of set partitions.
If $\rmX$ corresponds to the partition $\{S_0 , S_1 , \opp{S_1} ,\dots,S_k,\opp{S_k}\}$, then
\[
(\barr_n)^\rmX \cong \barr_k
\qqand
(\barr_n)_\rmX \cong \barr_{|S_0|/2} \times \arr_{|S_1|} \times \dots \times \arr_{|S_k|}.
\]

\begin{proof}[Proof of \cref{t:PEulB-recursion}]
Let $n \geq 1$ and $\rmH \in \cB_n$ be the hyperplane $x_n = 0$.
Then, $\cB_n^{\rmH} \cong \cB_{n-1}$.
The lines not in $\rmH$ correspond to partitions $\{S_0,S_1,\opp{S_1}\}$ with $n \notin S_0$.
If the line $\rmL$ corresponds to such partition, then
\[
(\cB_n)_{\rmL} \cong \barr_{|S_0|/2} \times \arr_{|S_1|}.
\]
It follows from \cref{eq:prod-arr,p:PEul-recursion} that
\[
P_{\cB_n}(z) = (z-1)P_{\cB_{n-1}}(z) + \sum P_{\cB_{n-1-k}}(z)P_{\arr_{k+1}}(z),
\]
where the sum is over partitions $\{S_0,S_1,\opp{S_1}\}$ with $n \in S_1$ and $|S_1| = k + 1$.
Such a partition is completely determined by the set $S_1 \setminus \{n\}$, which can be any involution-exclusive $k$ subset of $[\pm (n-1)]$.
The result follows since there are exactly $\binom{n-1}{k}2^k$ such subsets, and
\[
	\binom{n-1}{0}2^0 P_{\cB_{n-1-0}}(z)P_{\arr_{0+1}}(z) = P_{\cB_{n-1}}(z). \qedhere
\]
\end{proof}

We conclude this section with another result on the primitive analogue of a relation between the type A and type B Eulerian polynomials.
An explicit computation shows that
\[
\dfrac{\partial}{\partial x}A(z,x) = e^{x(z-1)}A(z,x)^2,
\]
where $A(x,z)$ is the generating function for the Eulerian polynomials in \cref{eq:generating-Eul-A}.
The following result is obtained by comparing the function above with the generating function of the type~B Eulerian polynomials.

\begin{proposition}For all $n \geq 0$,
\[
\sum_{k=0}^n \binom{n}{k} E_{\cB_k}(z) E_{\cB_{n-k}}(z) = 2^n E_{\arr_{n+1}}(z).
\]
\end{proposition}

In a similar manner, we can verify the following identity between the generating functions for the type A and type B Primitive Eulerian polynomials in \cref{t:PEul-generating}:
\[
\bigg( e^{x(z-1)} A(z,2x)^{1/2} \bigg)^2 = \dfrac{\partial ( 1 + \log A )}{\partial x} (z,2x).
\]

\begin{proposition}
For all $n \geq 0$,
\[
\sum_{k=0}^n \binom{n}{k} P_{\cB_k}(z) P_{\cB_{n-k}}(z) = 2^n P_{\arr_{n+1}}(z).
\]
\end{proposition}

In view of \cref{t:PE-A-des}, \cref{p:PEulB-1_2Eul}, and the symmetry of the coefficients of the Eulerian polynomial, this formula is equivalent to a result of Ma and Yeh \citep[Proposition 1]{my17EulStir}.

\section{The type D Primitive Eulerian polynomial}\label{ss:PEul-D}

The \defn{type D Coxeter arrangement} $\cD_n$ in $\RR^n$ consists of the hyperplanes with equations $x_i = x_j$ and $x_i = - x_j$ for all $1 \leq i < j \leq n$.
It is a subarrangement of $\barr_n$, and contains the braid arrangement $\arr_n$.
\[
\begin{gathered}
\begin{tikzpicture}[<->]
\node [] at (-2.2,1.3) {$\cD_2$};
\draw (45+180:1.5) -- (45:1.5) node [above right] {\small$x_1 = x_2$};
\draw (-45+180:1.5) -- (-45:1.5) node [below right] {\small$x_1 = -x_2$};
\end{tikzpicture}
\end{gathered}
\hspace*{.2\linewidth}
\begin{gathered}
\begin{tikzpicture}
\node [] at (0,1.3) {$\cD_3$};
\node [] at (0,-1.3) {$\phantom{\cD_3}$};
\end{tikzpicture}
\includegraphics[scale=.12]{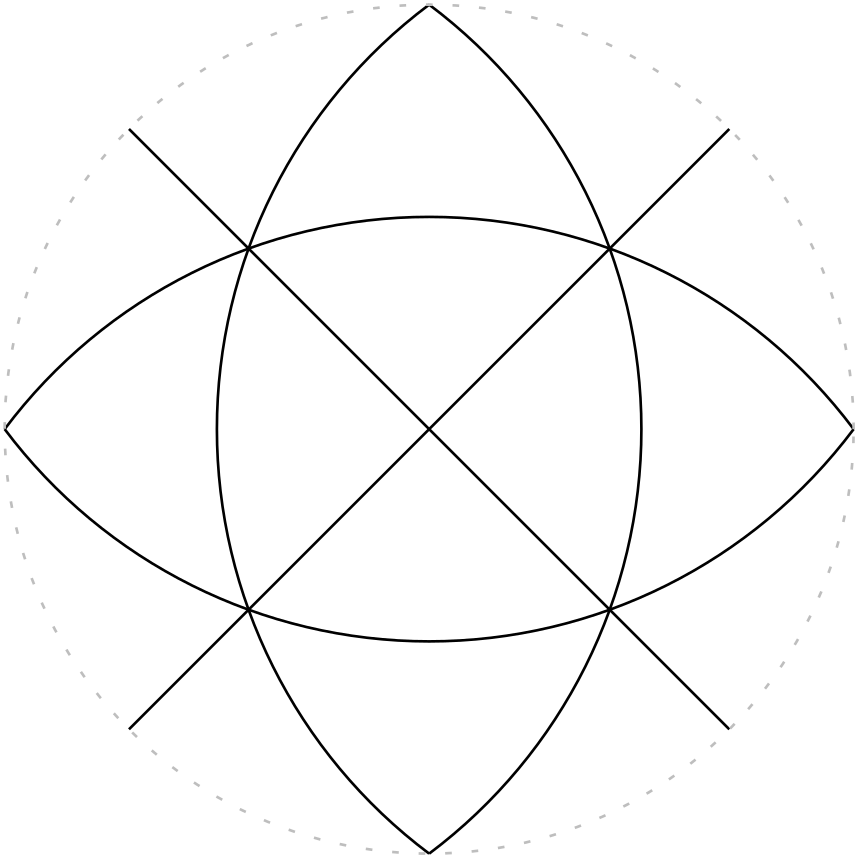}
\end{gathered}
\]
It is the reflection arrangement corresponding to the group of \defn{even signed permutations} $\frakD_n$, the Coxeter group of type ${\rm D}_n$.
$\frakD_n$ is the subgroup of $\frakB_n$ consisting of those signed permutations $w$ such that $\fneg(w)$ is even.
The descent statistic in $\frakD_n$ has the following combinatorial interpretation:
\begin{align*}
\des_D(w) & = \# \{ i \in [0,n-1] \,:\, w(i) > w(i+1)\} & \text{where } w(0) := -w(2).
\end{align*}

Unlike the type A and B cases, no combinatorial interpretation of the coefficients of $P_{\darr_n}(z)$ was given in \citep{bastidas20polytope}.
The obstacle being that no known excedance-like statistic on $\frakD_n$ has the right distribution on cuspidal elements.
In fact, none of the excedance statistics in \cref{ss:PEul-B} restricts to an Eulerian statistic on $\frakD_n$.
Another excedance-like statistic on $\frakD_n$, the \emph{flag weak excedance of type D} ${\rm fwex}_D$, was considered by Cho and Park \citep{cp18tqEulBD}, but this statistic is not Eulerian.

Using \cref{t:sharp-distr}, we will for the first time interpret the coefficients of $P_{\darr_n}(z)$ in combinatorial terms.
We identify the elements of $\frakD_n$ and the regions of the type B Coxeter arrangement $\darr_n$ as follows:
\[
w = w_1 w_2 \dots w_n \in \frakD_n
\quad
\longleftrightarrow
\quad
C_w := \{ x \in \RR^n \,:\, |x_{w_1}| \leq x_{w_2} \leq \dots \leq x_{w_n} \} \in \cham[\darr_n].
\]
Note that the region of $\darr_n$ associated to $w_1 w_2 \dots w_n \in \frakD_n$ is the union of the regions of $\barr_n$ corresponding to $w_1 w_2 \dots w_n \in \frakB_n$ and $\opp{w_1} w_2 \dots w_n \in \frakB_n$.
The weak order of $\cham[\darr_n]$ with respect to the region $C_e$ coincides with the weak order of $\frakD_n$ as a Coxeter group. In particular, $\des_D(w) = \des_\preceq(C_w)$.

Let $\sfh = \sfh_v^-$ for $v = (1,2,4,\dots,2^{n-1})$ be the halfspace we considered in the previous section.
Then, $C_w \subseteq \sfh$ if and only if all the right-to-left maxima of $|w| := |w_1| |w_2| \dots |w_n| \in \frakS_n$ are negative in $w$ and $|w_1| \neq n$.
See \citep[Proposition 8.3]{bw04geombases}.
We let $BW^D_n$ denote the collection of such elements.
The vector $v$ is very generic for $\cB_n$, and therefore it is very generic for any of its (essential) subarrangements; in  particular it is very generic for $\cD_n$. 
Since $v \in C_e$, the following is a direct consequence of \cref{t:sharp-distr}.

\begin{theorem}\label{t:PE-D-des}
For all $n \geq 2$,
\[
	P_{\cD_n}(z) = \sum_{w \in BW^D_n} z^{\des(w)}.
\]
\end{theorem}

\begin{figure}[ht]
\includegraphics[height=.2\textheight]{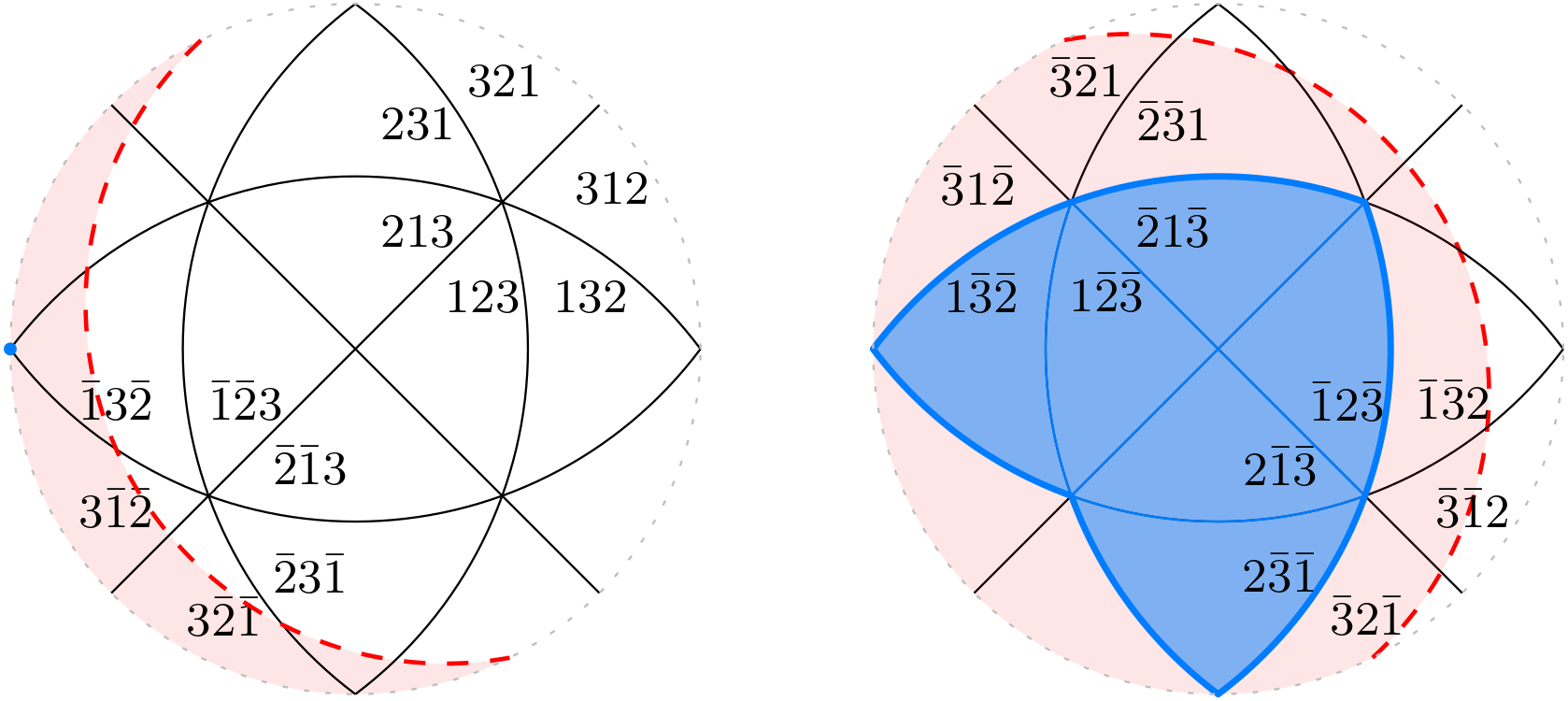}
\caption{Regions corresponding to elements in $BW^D_3$.}
\label{f:BWD3}
\end{figure}

\begin{example}
The following tables show the distribution of $\des_D$ on $BW^D_3$.
\begin{center}
{\renewcommand{\arraystretch}{1.1}
\begin{tabular}{c|c}
$w$ & $\des_D(w)$ \\
\hline
$1 \, \opp{3} \, \opp{2}$ & $3$ \\
$1 \, \opp{2} \, \opp{3}$ & $2$ \\
$\opp{2} \, 1 \, \opp{3}$ & $2$
\end{tabular}
\qquad\qquad
\begin{tabular}{c|c}
$w$ & $\des_D(w)$ \\
\hline
$2 \, \opp{1} \, \opp{3}$ & $2$ \\
$2 \, \opp{3} \, \opp{1}$ & $2$ \\
$\opp{1} \, 2 \, \opp{3}$ & $1$
\end{tabular}}
\end{center}
Thus, $P_{\darr_3}(z) = z^3 + 4 z^2 + z$.
\end{example}

Note that $P_{\darr_3}(z)$ is equal to the Primitive Eulerian polynomial of the braid arrangement $\arr_4$.
This is not surprising since $\darr_3$ is isomorphic to (the essentialization of) $\arr_4$.
However, the combinatorics of the complexes $BW^A_4$ and $BW^D_3$ (\cref{f:BWA4,f:BWD3}) are very different.
Indeed, the complex $BW^A_n$ is always a \emph{top-star}, meaning that it consists of all the regions containing a fixed face, in this case the face corresponding to the set composition $([n-1],\{n\})$. This is not true for the complex $BW^D_3$.

The following table shows the Primitive Eulerian polynomials of type D for the first values of $n$.

\begin{table}[ht]
\begin{tabular}{c||l}
$n$ & $P_{\darr_n}(z)$ \\
\hline\hline
$2$ & $z^2$ \\
$3$ & $z^3 + 4 z^2 + z$ \\
$4$ & $z^4 + 20 z^3 + 20 z^2 + 4 z$ \\
$5$ & $z^5 + 76 z^4 + 216 z^3 + 116 z^2 + 11 z$ \\
$6$ & $z^6 + 262 z^5 + 1732 z^4 + 2072 z^3 + 632 z^2 + 26 z$ \\
$7$ & $z^7 + 862 z^6 + 11824 z^5 + 28064 z^4 + 18404 z^3 + 3158 z^2 + 57 z$
\end{tabular}
\caption{The first Primitive Eulerian polynomials of type D.}
\end{table}

We now proceed to establish a quadratic recursion for the polynomials $P_{\darr_n}(z)$ in terms of the Primitive Eulerian polynomial of reflection arrangements of lower rank.
A key difference with the type A and B cases is that the arrangement under a hyperplane of $\darr_n$ is not (combinatorially isomorphic to) a reflection arrangement (i.e. $\darr_n$ it is not a \emph{good reflection arrangement} in the language of Aguiar and Mahajan \citep[Section 5.7]{am17}).

\begin{theorem}\label{t:PEulD-quad-rec}
The type D Primitive Eulerian polynomials satisfy the following recursion.
With $P_{\cD_0}(z) = 1$ and $P_{\cD_1}(z) = 0$,
\begin{align*}
& P_{\darr_n}(z) = (z-1)^2P_{\barr_{n-2}}(z) + \\
& \sum_{k=0}^{n-2} \binom{n-2}{k} 2^k 
								\Big( (z-1)P_{\cD_{n-2-k}}(z) P_{\arr_{k+1}}(z) + 2P_{\cD_{n-1-k}}(z) P_{\arr_{k+1}}(z) + P_{\cD_{n-2-k}}(z) P_{\arr_{k+2}}(z) \Big),
\end{align*}
for all $n \geq 2$.
\end{theorem}

We are not aware of a similar recursion for the type D Eulerian polynomials $E_{\darr_n}(z)$.
Since $\cD_n$ is a subarrangement of $\cB_n$, the flats of $\cD_n$ are also flats of $\cB_n$.
Specifically, the flats of $\cB_n$ that are flats of $\cD_n$ are those corresponding to type B partitions of $[\pm n]$ such that the zero block $S_0$ does not have cardinality $2$; see for instance Mahajan's thesis \citep[Appendix B]{mahajan02thesis}.
If the flat $\rmX \in \flat[\cD_n]$ corresponds to the partition $\{S_0 , S_1 , \opp{S_1} ,\dots,S_k,\opp{S_k}\}$, then
\[
(\darr_n)^\rmX \cong \begin{cases}
\darr_{k,r} & \text{if } |S_0| = 0, \\
\barr_k & \text{if } |S_0| \neq 0,
\end{cases}
\qqand
(\darr_n)_\rmX \cong \darr_{|S_0|/2} \times \arr_{|S_1|} \times \dots \times \arr_{|S_k|},
\]
where $\darr_{k,r}$ is the arrangement obtained by adding the first $r$ coordinate hyperplanes to $\darr_k$ and $r$ is the number of blocks $S_i$ with $|S_i| > 1$; not counting $S_i$ and $\opp{S_i}$ twice.
We will compute the Primitive Eulerian polynomial of the arrangements $\cD_{k,r}$ in \cref{ss:betweenDB}.

\begin{proof}[Proof of \cref{t:PEulD-quad-rec}]
Let $n \geq 2$ and $\rmH \in \darr_n$ be the hyperplane $x_1 = x_2$.
Then,
\[
\darr_n^\rmH \cong \cD_{n-1,1} = \cD_{n-1} \cup \big\{ \{ x \in \RR^{n-1} \,:\, x_1 = 0 \} \big\}.
\]
The lines not in $\rmH$ correspond to partitions $\{S_0,S_1,\opp{S_1}\}$ such that $1,2$ are not in the same block.
If~$\rmL$ corresponds to such partition, then
\[
(\darr_n)_{\rmL} \cong \cD_{|S_0|/2} \times \arr_{|S_1|}.
\]
These partitions come in three flavors:
\begin{center}
\hfill $1 \in S_0$ and $2 \in S_1$, or
\hfill $2 \in S_0$ and $1 \in S_1$, or
\hfill $1,\opp{2} \in S_1$.
\hfill \phantom{.}
\end{center}
It follows from \cref{p:PEul-recursion} that
\[
P_{\darr_n}(z) = (z-1)P_{\darr_n^\rmH}(z) + \sum P_{\cD_{n-1-k}}(z)P_{\arr_{k+1}}(z) + \sum P_{\cD_{n-2-k}}(z)P_{\arr_{k+2}}(z),
\]
where the first sum is over partitions of the first two kinds and $|S_1| = k + 1$,
and the second sum is over partitions of the third kind and $|S_1| = k + 2$.
Such partitions are completely determined by the set $S_1 \setminus \{2\}$, $S_1 \setminus \{1\}$, or $S_1 \setminus \{1,\opp{2}\}$, respectively; which can be any involution-exclusive $k$ subset of $[\pm n] \setminus [\pm 2]$ (as long as $|S_0| \neq 2$).
That is,
\[
P_{\darr_n}(z) = (z-1)P_{\darr_n^\rmH}(z) + \sum_{k=0}^{n-2} \binom{n-2}{k} 2^k \Big( 2 P_{\cD_{n-1-k}}(z)P_{\arr_{k+1}}(z) + \sum P_{\cD_{n-2-k}}(z)P_{\arr_{k+2}}(z) \Big).
\]
Observe that by setting $P_{\cD_1}(z) = 0$, we have implicitly taken care of the cases where $\{S_0,S_1,\opp{S_1}\}$ is not a partition of type D.
We now proceed to express $P_{\darr_n^\rmH}(z)$ in terms of Primitive Eulerian polynomials for reflection arrangements of lower rank.

Let $\rmH' \in \darr_n^\rmH$ be the hyperplane obtained by intersecting $x_1 = - x_2$ with $\rmH$.
Then,
\[
(\darr_n^\rmH)^{\rmH'} = \darr_n^{\rmH \cap \rmH'} \cong \cB_{n-2}.
\]
The lines contained in $\rmH$ and not contained in $\rmH'$ correspond to partitions where $1,2$ are in the same block but $1,\opp{2}$ are not. That is, partitions of the form $\{S_0,S_1,\opp{S_1}\}$ with $1,2 \in S_1$. 
In this case,
\[
(\darr_n)_\rmL^\rmH \cong  \cD_{|S_0|/2} \times \arr_{|S_1|-1}.
\]
Again, such partition is completely determined by $S_1 \setminus \{1,2\}$, so
\[
P_{\darr_n^\rmH}(z) = (z-1)P_{\barr_{n-2}}(z) + \sum_{k=0}^{n-2} \binom{n-2}{k} 2^k P_{\cD_{n-2-k}}(z)P_{\arr_{k+1}}(z).
\]
The result follows by substituting this into the expression for $P_{\darr_n}(z)$ above.
\end{proof}

\subsection{Generating function}

This section completes the proof of \cref{t:PEul-generating}.
We employ the following type B analog of the compositional formula.

\begin{proposition}\label{p:typeB-comp}
Let
\[
f(x) = \sum_{n \geq 0} f_n \dfrac{x^n}{n!} \qquad
g(x) = 1 + \sum_{n \geq 1} g_n \dfrac{x^n}{n!} \qquad
a(x) = \sum_{n \geq 1} a_n \dfrac{x^n}{n!}.
\]
If
\[
h(x) = \sum_{n \geq 0} h_n \dfrac{x^n}{n!}
\quad \text{where} \quad
h_n = \sum_{\{ S_0 , S_1,\opp{S_1}, \dots , S_k , \opp{S_k} \} \vdash^B [\pm n]} f_{|S_0|/2} g_k a_{|S_1|} \dots a_{|S_k|},
\]
then
\[
h(x) = f(x)g\big(\tfrac{a(2x)}{2}\big).
\]
\end{proposition}
Notice that we have adapted \citep[Proposition 6.3]{bastidas20polytope} for exponential generating functions, and to allow $f_0$, and therefore $h_0 = f_0$, to be different to 1.

\begin{proof}[Proof of \cref{t:PEul-generating}]
Recall the identification between the flats of the arrangement $\darr_n$ and the type D partitions of $[\pm n]$ from the previous section.
Let $\rmX \in \flat[\cD_n]$ correspond to the partition $\{ S_0 , S_1,\opp{S_1}, \dots , S_k , \opp{S_k} \}$. Then,
\[
\mu(\bot,\rmX) = \begin{cases}
(-1)^k (2k-3)!!(k+r-1) & \text{if } |S_0| = 0, \\
(-1)^k (2k-1)!! & \text{if } |S_0| \neq 0,
\end{cases}
\]
where $r$ denotes the number of pairs of blocks that are \underline{not singletons}, and $(-1)!! := 1$.
See for instance \citep[Example 2.3]{jt84free}.
Observe that in the first case, we necessarily have $k \geq 1$.

We introduce the auxiliary polynomials
\[
\Psi^{0}_{\cD_n}(z) = \sum (2k-3)!!(k+r-1) z^k
\qquad\qquad
\Psi^{>0}_{\cD_n}(z) = \sum (2k-1)!! z^k,
\]
where the first sum is over partitions $\{ S_0 , S_1,\opp{S_1}, \dots , S_k , \opp{S_k} \} \vdash^D [\pm n]$ with $|S_0| = 0$ and the second is over partitions with $|S_0| \neq 0$.
Observe that $\Psi_{\cD_n}(z) = \Psi^{0}_{\cD_n}(z) + \Psi^{>0}_{\cD_n}(z)$.
Now consider the bivariate polynomial
\[
\Phi_n(y,z) = \sum (2k-3)!! y^{r} (yz)^k,
\]
where the sum is over the same partitions defining $\Psi^{0}_{\cD_n}(z)$. Then,
\[
\dfrac{\partial}{\partial y} \dfrac{\Phi_n(y,z)}{y} \bigg|_{y=1} = \Psi^{0}_{\cD_n}(z).
\]
Using \cref{p:typeB-comp} with
\begin{align*}
f(x) = 1,\qquad
g(x) & = 1 + \sum_{n \geq 1}(2d-3)!! \dfrac{x^n}{n!} = 2 - (1-2x)^{1/2}, \quad \text{and}\\
a(y,z,x) & = yzx + \sum_{n \geq 2} y^2z \dfrac{x^n}{n!} = yzx + y^2z(e^x-x-1),
\end{align*}
we obtain that
\[
1 + yzx + \sum_{n \geq 2} \Phi_n(y,z) \dfrac{x^n}{n!} = 2 - \big( 1 - 2yzx - y^2z(e^{2x}-2x-1) \big)^{1/2}.
\]
Moving the constant term and dividing by $y$
\[
- \dfrac{1}{y} + zx + \sum_{n \geq 2} \dfrac{\Phi_n(y,z)}{y} \dfrac{x^n}{n!} = - \bigg( \dfrac{1}{y^2} - \dfrac{2zx}{y} - z(e^{2x}-2x-1) \bigg)^{1/2}.
\]
and differentiating with respect to $y$ and setting $y = 1$
\begin{equation}\label{eq:genD-aux1}
1 + \sum_{n \geq 2} \Psi^0_{D_n}(z) \dfrac{x^n}{n!} = \big( 1 - zx \big)\big( 1 - z(e^{2x}-1) \big)^{-1/2}.
\end{equation}
Recall that if $|S_0| \neq 0$, then $|S_0|/2 \geq 2$.
Therefore, using \cref{p:typeB-comp} with
\begin{align*}
f(x) = \sum_{n \geq 2}1 \dfrac{x^n}{n!} = e^x - x - 1,\qquad
g(x) & = 1 + \sum_{n \geq 1}(2n-1)!! \dfrac{x^n}{n!} = (1-2x)^{-1/2}, \quad \text{and}\\
a(z,x) & = \sum_{n \geq 1} z \dfrac{x^n}{n!} = z(e^x-1),
\end{align*}
we have
\begin{equation}\label{eq:genD-aux2}
\sum_{n \geq 2} \Psi^{>0}_{D_n}(z) \dfrac{x^n}{n!} = \big( e^x - x - 1 \big)\big( 1 - z(e^{2x}-1) \big)^{-1/2}.
\end{equation}
Adding \cref{eq:genD-aux1} and \cref{eq:genD-aux2},
\[
1 + \sum_{n \geq 2} \Psi_{\cD_n}(z) \dfrac{x^n}{n!} = \big( e^x - x(1+z) \big)\big( 1 - z(e^{2x}-1) \big)^{-1/2},
\]
and
\[
1 + \sum_{n \geq 2} P_{\cD_n}(z) \dfrac{x^n}{n!} = \big( e^{x(z-1)} - zx \big)\Big( \dfrac{z-1}{z - e^{2x(z-1)}} \Big)^{1/2},
\]
as we wanted to show.
\end{proof}

\section{The Primitive Eulerian polynomial of arrangements between type B and D}\label{ss:betweenDB}

Let $\cD_{n,k}$ be the arrangement obtained from $\darr_n$ by adding $k$ coordinate hyperplanes.
Note that the isomorphism class of $\cD_{n,k}$ does not depend on the particular coordinate hyperplanes we choose,
since any two choices are equal modulo the action of $\frakS_n \subseteq \frakD_n$.

Since $\cD_{n,k}$ is a subarrangement of $\barr_n$, faces of $\cD_{n,k}$ are the union of some of the faces of $\barr_n$.
Faces of $\barr_n$ are in correspondence with type B compositions of $[\pm n]$;
these are \emph{ordered} type B partitions of the form
$(\opp{S_k},\dots,\opp{S_1},S_0,S_1,\dots,S_k)$. The corresponding face of $\barr_n$ is determined by inequalities $x_a \leq x_b$ whenever the block containing $a$ weakly precedes the block containing $b$;
in particular, the face is contained in the hyperplanes $x_a = 0$ for all $a \in S_0$ and in the hyperplanes $x_a = x_b$ whenever $a$ and $b$ are in the same block of the composition.

Faces of $\cD_{n,0} = \cD_n$ are described by Mahajan in \citep[Appendix B.6]{mahajan02thesis}.
They are either
\begin{itemize}
\item faces of $\barr_n$ corresponding to a compositions $(\opp{S_k},\dots,\opp{S_1},S_0,S_1,\dots,S_k)$ with $|S_0| \geq 4$,
\item faces of $\barr_n$ corresponding to a compositions $(\opp{S_k},\dots,\opp{S_1},\emptyset,S_1,\dots,S_k)$ with $|S_1| \geq 2$, or
\item the union of exactly three faces of $\barr_n$ corresponding to compositions
	\begin{itemize}
	\item $(\opp{S_k},\dots,\opp{S_2},\{\opp{a},a\},S_2,\dots,S_k)$,
	\item $(\opp{S_k},\dots,\opp{S_2},\{\opp{a}\},\emptyset,\{a\},S_2,\dots,S_k)$, and
	\item $(\opp{S_k},\dots,\opp{S_2},\{a\},\emptyset,\{\opp{a}\},S_2,\dots,S_k)$,
	\end{itemize}
	for some $a \in [n]$.
\end{itemize}
Faces of the first two types are faces of all arrangements $\cD_{n,k}$.
A face of the third type is a face of $\cD_{n,k}$ if and only if the hyperplane $x_a = 0$ is not in the arrangement $\cD_{n,k}$,
otherwise the corresponding three faces of $\barr_n$ are faces of $\cD_{n,k}$.

\begin{theorem}\label{t:PEul-D-to-B}
For all $0 \leq k \leq n$, the Primitive Eulerian polynomial of $\cD_{n,k}$ is given by the following formula.
\[
P_{\cD_{n,k}}(z) = P_{\darr_n}(z) + k z^n P_{\cB_{n-1}}(\tfrac{1}{z}).
\]
\end{theorem}

\begin{proof}
Since $\barr_n = \cD_{n,n}$, the result follows from the following two claims.
\begin{enumerate}[(i.)]
\item For all $n \geq 0$, $P_{\cB_n}(z) = P_{\darr_n}(z) + n z^n P_{\cB_{n-1}}(\tfrac{1}{z})$.
\item For fixed $n \geq 0$, the difference $P_{\cD_{n,k+1}}(z) - P_{\cD_{n,k}}(z)$ is the same for all $k = 0,\dots,n-1$.
\end{enumerate}

We prove (i.) by using the power series of \cref{t:PEul-generating}.
We first compute the generating function for $n z^n P_{\cB_{n-1}}(\tfrac{1}{z})$:
\begin{align*}
\sum_{n \geq 0} n z^n P_{\cB_{n-1}}(\tfrac{1}{z}) \dfrac{x^n}{n!} 
= zx \sum_{n \geq 1} P_{\cB_{n-1}}(\tfrac{1}{z}) \dfrac{(zx)^{n-1}}{(n-1)!} 
= zx e^{zx(\tfrac{1}{z}-1)} \Big( \dfrac{\tfrac{1}{z}-1}{\tfrac{1}{z}-e^{2zx(\tfrac{1}{z}-1)}}  \Big)^{1/2} \\
= zx e^{x(1-z)} \Big( \dfrac{1-z}{1 - ze^{2x(1-z)}}  \Big)^{1/2}
= zx A(z,2x)^{1/2} .
\end{align*}
The relation $P_{\cB_n}(z) = P_{\darr_n}(z) + n z^n P_{\cB_{n-1}}(\tfrac{1}{z})$ for all $n$ is thus equivalent to the following equality between the corresponding generating functions:
\[
e^{x(z-1)} A(z,2x)^{1/2} = \big( e^{x(z-1)} - z x \big) A(z,2x)^{1/2} + zx A(z,2x)^{1/2} .
\]

We proceed to prove (ii.)
Let $\sfh = \sfh_v^-$ for $v = (1,2,4,\dots,2^{n-1})$, a generic halfspace for all $\cD_{n,k}$.
In view of \cref{eq:cochar-faces-arr}, it suffices to show that the difference between $f$-polynomial of $\face_{\cD_{n,k+1}}(\sfh)$ and $\face_{\cD_{n,k}}(\sfh)$ is independent of $k \in [0,n-1]$.
Without loss of generality, assume
$\cD_{n,k}$ is $\cD_{n}$ with the first $k$ coordinate arrangements, and that $\cD_{n,k+1}$ is obtained by adding the hyperplane $x_n = 0$ to $\cD_{n,k}$.
By the description of the faces of $\cD_{n,k}$ above,
$\face[\cD_{n,k+1}]$ is obtained from $\face[\cD_{n,k}]$ by removing the faces of the third type with $a = n$ and adding the corresponding faces of $\barr_n$.
Which of these faces are contained in $\sfh$ (both the ones removed and added) is independent of $k$.
\end{proof}

\section{Real-rootedness}\label{s:roots}

\subsection{Real-rootedness in rank at most 3}
In this section, we employ the theory of interlacing polynomials to prove that the Primitive Eulerian polynomial of any arrangement of rank at most $3$ is real-rooted.

Let $f(z),g(z)$ be real-rooted polynomials of degree $d$ and $d-1$ respectively. The polynomial $g$ \defn{interlaces} $f$ if
\[
\alpha_1 \leq \beta_1 \leq \alpha_2 \leq \dots \leq \beta_{d-1} \leq \alpha_d,
\]
where $\alpha_1,\dots,\alpha_d$ and $\beta_1,\dots,\beta_{d-1}$ are the roots of $f$ and $g$, respectively. In this case, the polynomial $f(z)+g(z)$ is real-rooted. See for example \citep[Theorem 8]{branden04real}.

\begin{theorem}\label{t:rk3-real}
Let $\arr$ be an arrangement of rank $r \leq 3$.
Then, $P_\arr(z)$ is real-rooted.
\end{theorem}

\begin{proof}
The cases $r = 1,2$ follow from the explicit computations in \cref{ex:PEul-rk1,ex:PEul-rk2}.

Let $\arr$ be an essential arrangement of $n$ hyperplanes in $\RR^3$.
Fix a hyperplane $\rmH_0 \in \arr$, and let $k$ (resp. $m$) be the number of lines of $\arr$ contained in $\rmH_0$ (resp. not contained in $\rmH_0$).
\cref{p:PEul-recursion} reads
\[
P_\arr(z) = (z-1) P_{\arr^{\rmH_0}}(z) + \sum_{\rmL} P_{\arr_{\rmL}}(z).
\]
We will show that $\sum_{\rmL} P_{\arr_{\rmL}}(z)$ interlaces $(z-1) P_{\arr^{\rmH_0}}(z)$, and therefore $P_\arr(z)$ is real-rooted.

First, \cref{ex:PEul-rk2} shows that $P_{\arr^{\rmH_0}}(z) = z^2 + (k-2)z$,
so the polynomial $(z-1) P_{\arr^{\rmH_0}}(z)$ 
has zeros at
\[
z = \red{-(k-2)},\red{0},\red{1}.
\]
On the other hand,
\[
\sum_{\rmL} P_{\arr_{\rmL}}(z) = m z^2 + \sum_{\rmL} (n_{\rmL} - 2)  z,
\]
where the sum is over the lines not contained in $\rmH_0$ and $n_{\rmL}$ denotes the number of hyperplanes containing~${\rmL}$.
The zeros of this polynomial are
\[
z = \blue{- \big( \tfrac{\sum_{\rmL} n_{\rmL}}{m} - 2 \big)}, \blue{0}.
\]
\textbf{Claim:} For any $\rmL$ not contained in $\rmH_0$, $n_{\rmL} \leq k$.\\
To prove the claim, note that for all pair of distinct hyperplanes $\rmH , \rmH' \in \arr$ containing $\rmL$, the lines $\rmH \cap \rmH_0$ and $\rmH' \cap \rmH_0$ are distinct. Otherwise, a dimension argument shows that
\[
	\rmL = \rmH \cap \rmH' = \rmH \cap \rmH' \cap \rmH_0 \subseteq \rmH_0.
\]
Since $\rmH_0$ contains $k$ lines, no more than $k$ distinct hyperplanes of $\arr$ can contain $\rmL$.

Therefore we have $2 \leq n_{\rmL} \leq k$ for all ${\rmL}$, and $0 \leq \tfrac{\sum_{\rmL} n_{\rmL}}{m} - 2 \leq k - 2$.
It follows that $\sum_{\rmL} P_{\arr_{\rmL}}(z)$ interlaces $(z-1) P_{\arr^{\rmH_0}}(z)$:
\[
\red{-(k-2)} \leq \blue{- \big( \tfrac{\sum_{\rmL} n_{\rmL}}{m} - 2 \big)} \leq \red{0} \leq \blue{0} \leq \red{1},
\]
as we wanted to show.
\end{proof}

\begin{remark}
Note that this result does not assume the arrangement to be simplicial.
For instance, the Primitive Eulerian polynomial of the graphic arrangement of a $4$-cycle in \cref{ex:PEul-4cycle} is real-rooted:
\[
z^3 + 3z^2 - z = z(z-\tfrac{-3+\sqrt{13}}{2})(z-\tfrac{-3-\sqrt{13}}{2}).
\]
\end{remark}

\begin{remark}
It immediately follows from \cref{t:rk3-real} that the cocharacteristic polynomial of any arrangement of rank at most $3$ is real-rooted. The same is not true for the characteristic polynomial.
For instance, the characteristic polynomial of the graphic arrangement of a $4$-cycle is $z(z-1)(z^2-3z+3)$, which has a pair of conjugate complex roots.
\end{remark}

\subsection{Real-rootedness fails non-simplicial arrangements in higher rank}

If $\arr$ is a non-simplicial arrangement of rank at least $4$, $P_\arr(z)$ might fail to be real-rooted.

\begin{example}
For $n \geq 2$, let $\cG_n$ be the arrangement of $n+1$ generic hyperplanes in $\RR^n$.
It makes sense to say \emph{the} arrangement, because any two such arrangements are combinatorially isomorphic; this is not true if we have more hyperplanes.
The arrangement $\cG_2$ is isomorphic to $\cI_2(3)$, and $\cG_3$ is isomorphic to the graphic arrangement of a $4$-cycle.

We verify using \cref{p:PEul-recursion} and induction that
\[
P_{\cG_n}(z) = z^n + \sum_{k = 1}^{n-1} (-1)^{k+1} \binom{n}{k+1} z^{n-k} = z(z-1)^n + (n+1)z^n - z^{n+1}.
\]
The second expression makes the inductive step cleaner, even though it involves terms of higher degree.
For $n = 2$, we have
\[
z(z-1)^2 + 3z^2 - z^3 = z^2 + z = P_{\cI_2(3)}(z).
\]

Now assume $n \geq 3$ and let $\rmH$ be any hyperplane in $\cG_n$.
The arrangement $\cG_n^\rmH$ consists of $n$ hyperplanes in generic position, so $\cG_n^\rmH \cong \cG_{n-1}$.
Moreover, any line in $\cG_n$ is contained in exactly $n-1$ hyperplanes; thus there are $\binom{(n+1)-1}{n-1} = n$ lines not contained in $\rmH$, and $(\cG_n)_\rmL$ is isomorphic to a coordinate arrangement of $n-1$ hyperplanes for each line $\rmL$.
Therefore,
\begin{align*}
P_{\cG_n}(z) & = (z-1)P_{\cG_{n-1}}(z) + nz^{n-1} \\
& = (z-1)(z(z-1)^{n-1} + nz^{n-1} - z^n) + nz^{n-1} \\
& = z(z-1)^n + (n+1)z^n - z^{n+1},
\end{align*}
as claimed.
Already for $n=4$, we have
$
P_{\cG_4}(z) = z^4 + 6 z^3 - 4 z^2 + z,
$
which has two complex roots.

\end{example}

\subsection{Real-rootedness for Coxeter and simplicial arrangements}

Frobenius \citep{frob1910Eulerian} first proved that the classical (type A) Eulerian polynomials are real-rooted.
In view of identity \cref{eq:typeA-Eul-is-PEul}, this implies that the type A Primitive Eulerian polynomials are real-rooted.
Later, Savage and Visontai \citep{sv15sEulerianRoots} proved that the polynomials $A^{(2)}_n(z)$ have only real roots.
It then follows by \cref{p:PEulB-1_2Eul} that so do the type B Primitive Eulerian polynomials.

With the help of SageMath \citep{sagemath}, we verified that Primitive Eulerian polynomials of the exceptional type in \cref{tab:PrimEulExceptional} are all real-rooted.
We have also verified that the type D Primitive Eulerian polynomial $P_{\cD_n}(z)$ is real-rooted for $n \leq 300$.
We conjecture this is true for all values of $n$,
and therefore that the Primitive Eulerian polynomial of any Coxeter arrangement is real-rooted.

\begin{conjecture}\label{c:PEul-Cox-rr}
The Primitive Eulerian polynomial of any real reflection arrangement is real-rooted.
\end{conjecture}

The corresponding conjecture for the Eulerian polynomial was originally posed by Brenti \citep{brenti94cox-eul} and solved two decades later by Savage and Visontai \citep{sv15sEulerianRoots}.

The arrangements $\cD_{n,k}$ are all simplicial, and not combinatorially isomorphic to reflection arrangements whenever $k \neq 0,n$.
We have computationally verified that $P_{\cD_{n,k}}(z)$ is real-rooted for all $k \leq n \leq 150$.
We have also verified that the Primitive Eulerian polynomial of the crystallographic simplicial arrangements of Cuntz and Heckenberger \citep{ch15Weyl}, and the two additional examples in rank $4$ by Geis \citep{geis2019simplicial}, are all real-rooted.
We conclude this section with the following conjecture, which would of course imply \cref{c:PEul-Cox-rr}.

\begin{conjecture}
The Primitive Eulerian polynomial of any simplicial arrangement is real-rooted.
\end{conjecture}

\bibliographystyle{plain}
{
\bibliography{bib}}

\end{document}